\documentclass[letterpaper]{article}

\usepackage{amsthm}
\usepackage{amsmath}
\usepackage{amsfonts}
\usepackage{mathtools}
\usepackage{indentfirst}
\usepackage{arydshln}
\usepackage{color}
\usepackage{setspace}
\usepackage{lineno}
\usepackage{enumitem}
\usepackage{algorithmic} 

\usepackage{graphicx}
\usepackage[normalem]{ulem}
\usepackage[toc,page]{appendix}
\usepackage[utf8]{inputenc}
\usepackage[english]{babel}
\usepackage{hyperref}
\usepackage{cleveref}

\usepackage{multirow}
\usepackage{booktabs}
\usepackage{layouts}
\usepackage{tikz}
\usetikzlibrary{shapes.geometric, arrows}

\newcommand{\optval}[1]{\mathcal{V}^*_{#1}} 
\usepackage{soul}

\makeatletter
\newcommand*\bigcdot{\mathpalette\bigcdot@{.5}}
\newcommand*\bigcdot@[2]{\mathbin{\vcenter{\hbox{\scalebox{#2}{$\m@th#1\bullet$}}}}}
\makeatother

\newcommand{\defeq}{:=}

\usepackage[section]{algorithm}

\DeclareMathOperator{\Tr}{Tr}
\DeclareMathOperator{\rank}{rank}

\numberwithin{equation}{section}

\theoremstyle{theorem} 
   \newtheorem{theorem}{Theorem}[section]
   \newtheorem{remark}[theorem]{Remark}
   \newtheorem{lemma}[theorem]{Lemma}
   \newtheorem{proposition}[theorem]{Proposition}
   \newtheorem{assumption}[theorem]{Assumption}
   \newtheorem{definition}[theorem]{Definition}

\usepackage{comment}
\newcommand{\firstRevisionColor}{\color{black}}
\newcommand{\SecondRevisionColor}{\color{black}}
\newcommand{\ThirdRevisionColor}{\color{black}}
\newcommand{\revisionColor}{\color{black}}

\title{An Optimality Gap Test for a Semidefinite Relaxation of a Quadratic Program with Two Quadratic Constraints\thanks{AFOSR grants FA 9550-14-1-0398 and FA9550-19-1-0315, and NSF grant ECCS~1408320.}}	
\author{Sheng Cheng \and Nuno C. Martins\thanks{Department of Electrical and Computer Engineering and the Institute for Systems Research, University of Maryland, College Park, MD, USA ({cheng@terpmail.umd.edu}, {nmartins@umd.edu}).}}	
\date{}

\begin{document}
	\maketitle
	
	\begin{abstract}
		We propose a necessary and sufficient test to determine whether a solution for a general quadratic program with two quadratic constraints (QC2QP) can be computed from that of a specific convex semidefinite relaxation, in which case we say that there is no optimality gap. Originally intended to solve a nonconvex optimal control problem, we consider the case in which the cost and both constraints of the QC2QP may be nonconvex. We obtained our test, which also ascertains when strong duality holds, by generalizing a closely-related method by {Ai~and~Zhang}. An extension was necessary because, while the method proposed by {Ai~and~Zhang} also allows for two quadratic constraints, it requires that at least one is strictly convex. In order to illustrate the usefulness of our test, we applied it to two examples that do not satisfy the assumptions required by prior methods. Our test guarantees that there is no optimality gap for the first example---a solution is also computed from the relaxation---and we used it to establish that an optimality gap exists in the second. 
	\end{abstract}
	
	\providecommand{\keywords}[1]	
	{	
		\small		
		\textbf{\textit{Keywords. }} #1	
	}	
	\keywords{quadratically constrained quadratic program, semidefinite relaxation, strong duality, nonconvex optimization}	
	
	\providecommand{\AMS}[1]	
	{	
		\small		
		\textbf{\textit{AMS subject classifications. }} #1	
	}	
	\AMS{90C26, 90C46}	
	
	
	\section{Introduction}
	
	We consider the following real-valued quadratic program with two quadratic constraints (QC2QP):
	\begin{equation}
	\begin{aligned}
	& \underset{\boldsymbol{z} \in \mathbb{R}^n}{\text{minimize}}
	& & q_0(\boldsymbol{z}) = \boldsymbol{z}^T Q_0\boldsymbol{z} +2b_0^T\boldsymbol{z} \\
	& \text{subject to}
	& & q_1(\boldsymbol{z}) =   \boldsymbol{z}^T Q_1\boldsymbol{z} +2b_1^T\boldsymbol{z} + c_1 \leq 0,\\
	&
	& & q_2(\boldsymbol{z}) =   \boldsymbol{z}^T Q_2\boldsymbol{z} +2b_2^T\boldsymbol{z} + c_2 \leq 0,\\
	\end{aligned}
	\tag{QP0}\label{prob: very original problem}
	\end{equation}
	where $Q_0$, $Q_1$, and $Q_2$ are $n\times n$-dimensional real symmetric matrices; $b_0$, $b_1$ and $b_2$ are $n$-dimensional real vectors; and $c_1$ and $c_2$ are real constants. 
	
	\vspace{.1 in} 
	\noindent{\bf Main problem:} We seek to solve \eqref{prob: very original problem} without positive semidefiniteness restrictions on $Q_0, Q_1$, and $Q_2$, which, generally, makes the problem nonconvex. 
	\subsection{Brief overview of existing related work}
	Existing work explored two distinct approaches to obtain a globally optimal solution to \eqref{prob: very original problem}. The first approach, which we adopt to develop our method, uses a semidefinite relaxation of \eqref{prob: very original problem} whose (Lagrange) dual is convex and identical to that of \eqref{prob: very original problem}.  The second approach seeks to exploit the structure of the QC2QP, possibly subject to additional restrictions, to characterize globally optimal solutions in a way that numerically tractable methods can be used. Subsequently, we proceed to describe previous work on both approaches.
	
	Following the first approach, {Ai~and~Zhang}~\cite{ai2009strong} introduce a necessary and sufficient condition for strong duality for the Celis-Dennis-Tapia (CDT) subproblem of minimizing a nonconvex quadratic cost over the intersection of an ellipsoid and an elliptical cylinder~\cite{celis1985trust} (corresponding to $Q_1$ being positive definite and $Q_2$ being positive semidefinite, respectively, in \eqref{prob: very original problem}), which is a special case of QC2QP used in the extended trust region method \cite{ye2003new}. Their result shows that strong duality holds, and a primal optimal solution can be obtained from a semidefinite relaxation, if and only if optimal solutions of the dual and the relaxation violate the so-called \textit{Property~I} comprising three algebraic conditions. Subsequent work by Yuan~et~al. in~\cite{yuan2017new} shows that adding second-order cone (SOC) constraints to a CDT subproblem for which \textit{Property~I} holds may narrow or even eliminate the duality gap. In the latter case, a globally optimal solution to the original problem can be computed from a solution of the semidefinite relaxation with an SOC reformulation.
	
	Another relaxation technique is to solve the QC2QP in the complex domain. In~\cite{beck2006strong},  Beck~and~Eldar use such a methodology to introduce a necessary and sufficient condition for strong duality, using the classical extended S-Lemma of Fradkov and Yakubovich~\cite{fradkov1973s}. If strong duality holds, a globally optimal solution to the original problem can be obtained by solving a quadratic feasibility problem. Using the necessary and sufficient condition and the convexity of a quadratic mapping, they subsequently prove a sufficient condition for strong duality for the real-valued QC2QP. Huang and Zhang \cite{huang2007complex} propose a sufficient condition for strong duality in the complex-valued problem in which a globally optimal solution to the original problem can be obtained from a semidefinite relaxation if strong duality holds. Their result is derived using a matrix rank-one decomposition for complex Hermitian matrices. 
	
	Following the second approach, Peng and Yuan \cite{peng1997optimality} prove a necessary condition for global optimality in QC2QP. Specifically, the number of negative eigenvalues of the Hessian of the Lagrangian is characterized at a globally optimal solution. For the CDT subproblem, Bomze and Overton \cite{bomze2015narrowing} prove necessary and sufficient conditions for global and local optimality using copositivity. 
	
	
	\subsection{Our main contribution}
	We seek to use a specific semidefinite relaxation to find a solution for \eqref{prob: very original problem} for the case in which there are no positive semidefiniteness restrictions on $Q_0, Q_1$, and $Q_2$.
	When a solution for \eqref{prob: very original problem} can be determined from that of the relaxation we say that there is {\it no optimality gap}.  The relaxation is cast as a convex semidefinite program (SDP) for which an optimal solution can be determined efficiently using existing software. The dual of the semidefinite relaxation is convex and is also the dual of \eqref{prob: very original problem}. This motivates the analysis in \cref{chap: main results}, where we propose the so-called \textit{Property~I$^+$} defined by four algebraic conditions that determine, based on solutions of the relaxation and its dual, when an optimality gap exists. 	Our main result is \Cref{thm: necessary and sufficient condition for strong duality}, which states precisely a necessary and sufficient condition for the existence of an optimality gap based on \textit{Property~I$^+$}. As we discuss in detail in \cref{chap: comparison}, \Cref{thm: necessary and sufficient condition for strong duality} extends the closely-related result of~\cite{ai2009strong} in the following ways:
	
	\begin{itemize}
		\item The assumption in~\cite{ai2009strong} that either $Q_1$ or $Q_2$ must be positive definite is replaced in our work by the weaker requirement that the dual of  \eqref{prob: very original problem} satisfies Slater's condition.
		\item In the particular case when $Q_1$, or $Q_2$, is positive definite the above-mentioned \textit{Property~I$^+$} is equivalent to \textit{Property~I} used in~\cite{ai2009strong} to determine when there is an optimality gap. Hence, our work presents no advantage relative to~\cite{ai2009strong} when $Q_1$, or $Q_2$, is positive definite.
	\end{itemize}
	
	A nonconvex optimal control problem studied by Cheng and Martins in~\cite{chengmartins2018reaching} motivated the unexampled QC2QP considered here, in which neither $Q_1$ nor $Q_2$ is assumed positive definite. 
	
	
	\subsection{Structure of the article}	
	We start with reviewing in \cref{chap: preliminary results} the key results of~\cite{ai2009strong}. In doing so, we also present the essential concepts used in~\cite{ai2009strong}, which include the semidefinite relaxation used here. We define \textit{Property~I$^+$} and subsequently state our main result (\Cref{thm: necessary and sufficient condition for strong duality}) in \cref{chap: main results}.
	In \cref{chap: test}, we describe an algorithm to implement the test of \Cref{thm: necessary and sufficient condition for strong duality}, and we also discuss relevant numerical considerations. In addition, in~\cref{chap: test}, we apply our algorithm to two QC2QP examples that do not satisfy the assumptions required by previous methods. More specifically, we compute the optimal solution for the first example after we establish that it has no optimality gap. In contrast, we establish that there is an optimality gap for the second. 
	In \cref{chap: comparison}, we explain why \Cref{thm: necessary and sufficient condition for strong duality} extends the closely-related result of~\cite{ai2009strong}. {\firstRevisionColor In \cref{sec: conclusions and future work}, we summarize the paper and discuss future research directions that could leverage our results, in concert with related recent work, to obtain methods that would not only be numerically more robust, but also more general.} In \cref{sec: proof}, we present a detailed proof of~\Cref{thm: necessary and sufficient condition for strong duality}. {\firstRevisionColor The necessity portion of the proof embeds the description of a method to construct a solution to (QP0) from that of its convex relaxation, for the case in which the test in~\Cref{thm: necessary and sufficient condition for strong duality} guarantees that there is no optimality gap.}
	
	\subsection{Notation and conventions}
	
	Throughout the paper, we adopt the following notation, which is mostly borrowed from~\cite{ai2009strong}: We denote the set of real numbers with $\mathbb{R}$. We use $\mathcal{S}^n$ to denote the set of symmetric matrices in $\mathbb{R}^{n \times n}$. We use the dot notation to denote the matrix inner product, that is, $A \bigcdot B \defeq \Tr(AB^T) $ for $A,B \in \mathbb{R}^{n \times n}$, where $\Tr(AB^T)$ denotes the trace of $AB^T$. We use $\det(C)$ to denote the determinant of a square matrix $C$. We use $\rank(D)$ and $\rank(D,\epsilon)$ to denote the rank and the numerical rank with tolerance $\epsilon$, respectively, of a matrix $D$. A positive (semi)definite matrix $M$ is denoted by $M \succ (\succeq) 0$. We use $0_{n \times m}$ to denote a matrix in $\mathbb{R}^{n \times m}$ with all entries being $0$ and $I_n$ to denote an $n \times n$-dimensional identity matrix. A diagonal matrix is denoted by $\text{diag}(a_1,a_2,\dots,a_n)$, where $a_1,a_2,\dots,a_n \in \mathbb{R}$ are the diagonal entries. The null space of a linear mapping $L: \mathbb{V} \rightarrow \mathbb{W}$ between two vector spaces $\mathbb{V}$ and $\mathbb{W}$ is denoted by $\mathcal{N}(L)$. We use $|a|$ to denote the absolute value of a real-valued constant or variable $a$. We use the term \textit{polynomial time}, which is defined in \cite{ye2003new}, to indicate that the total number of basic operations (for example, addition, subtraction, multiplication, division, and comparison) of a procedure is bounded by a polynomial of the problem data. We use boldface font, such as in $\boldsymbol{x}$, to represent 
	the optimization variables with respect to which we seek to minimize a cost subject to constraints. We adopt the following format to represent an optimization problem over a subset $\mathcal{X}$ of a real coordinate space, in which we seek to minimize a cost $f:\mathcal{X} \rightarrow \mathbb{R}$ subject to an additional constraint set $\mathcal{C}$: 
	\begin{equation}
	\begin{aligned}
	& \underset{\boldsymbol{x} \in \mathcal{X}}{\text{minimize}}
	& & f(\boldsymbol{x})\\
	& \text{subject to}
	& & \boldsymbol{x} \in \mathcal{C}.
	\end{aligned}
	\tag{P}\label{prob: notation problem}
	\end{equation}
	We use $\optval{\eqref{prob: notation problem}}$ to denote the optimal value of~\eqref{prob: notation problem}. As a convention, a matrix $X$ is rank-one decomposable \cite{ai2009strong} at $x_1$ if there exist other $r-1$ vectors $x_2,\dots,x_r$ such that $X = x_1x_1^T + x_2x_2^T + \dots +x_r x_r^T$, where $r \defeq \rank(X)$.
	
	\section{Preliminary results and concepts}\label{chap: preliminary results}
	
	We start with introducing assumptions and reviewing the key results in \cite{ai2009strong}. For the reader's convenience, we follow the notation in \cite{ai2009strong} and rewrite \eqref{prob: very original problem} in a homogeneous quadratic form:
	\begin{equation}
	\begin{aligned}
	& \underset{\boldsymbol{z} \in \mathbb{R}^n, \boldsymbol{t} \in \mathbb{R}}{\text{minimize}}
	& & M(q_0) \bigcdot \begin{bmatrix}
	\boldsymbol{t} \\ 
	\boldsymbol{z}
	\end{bmatrix} 
	\begin{bmatrix}
	\boldsymbol{t} \\ 
	\boldsymbol{z}
	\end{bmatrix}^T =  \boldsymbol{z}^T Q_0\boldsymbol{z} +2\boldsymbol{t}b_0^T\boldsymbol{z}\\
	& \text{subject to}
	& & M(q_i) \bigcdot  \begin{bmatrix}
	\boldsymbol{t} \\ 
	\boldsymbol{z}
	\end{bmatrix} 
	\begin{bmatrix}
	\boldsymbol{t} \\ 
	\boldsymbol{z}
	\end{bmatrix}^T =  \boldsymbol{z}^T Q_i\boldsymbol{z} +2\boldsymbol{t}b_i^T\boldsymbol{z} + \boldsymbol{t}^2 c_i \leq 0, \qquad i \in \{1,2\},\\
	&
	&& \boldsymbol{t}^2 = 1,
	\end{aligned}
	\tag{QP}\label{prob: original  problem}
	\end{equation}
	where 
	\begin{equation}\label{eq: problem data}
	M(q_0) = \begin{bmatrix}
	0 & b_0^T \\
	b_0 & Q_0
	\end{bmatrix}, \quad M(q_i) = \begin{bmatrix}
	c_i & b_i^T \\
	b_i & Q_i
	\end{bmatrix}, \quad i \in  \{1,2\}.
	\end{equation}
	In the rest of the paper, we use $x$ to represent an $(n+1)$-dimensional vector concatenating a scalar $t$ and an $n$-dimensional vector $z$ as follows
	\begin{equation}\label{eq: partition of x}
	x = \begin{bmatrix}
	t \\ z
	\end{bmatrix}.
	\end{equation}
	
	The semidefinite relaxation of \eqref{prob: original  problem} is the following: 
	\begin{equation}
	\begin{aligned}
	& \underset{\boldsymbol{X} \in \mathcal{S}^{n+1}}{\text{minimize}}
	& & M(q_0) \bigcdot \boldsymbol{X} \\
	& \text{subject to}
	& & M(q_i) \bigcdot \boldsymbol{X} \leq 0, \qquad i \in \{1,2\},\\
	&
	&& I_{00} \bigcdot \boldsymbol{X} = 1,\\
	&
	&& \boldsymbol{X} \succeq 0,
	\end{aligned}
	\tag{SP}\label{prob: semidefinite relaxation}
	\end{equation}
	where $I_{00} = \bigl[\begin{smallmatrix}
	1 & 0_{1 \times n}\\
	0_{n \times 1} & 0_{n \times n}
	\end{smallmatrix} \bigr]$.
	
	The dual problem of \eqref{prob: semidefinite relaxation} is the following:
	\begin{equation}
	\begin{aligned}
	& \underset{\boldsymbol{Z} \in \mathcal{S}^{n+1},\boldsymbol{y}_0,\boldsymbol{y}_1,\boldsymbol{y}_2 \in \mathbb{R}}{\text{maximize}}
	& & \boldsymbol{y}_0 \\
	&  \ \ \ \ \text{subject to}
	&& \boldsymbol{y}_0 I_{00} - \boldsymbol{y}_1 M(q_1) - \boldsymbol{y}_2 M(q_2) + \boldsymbol{Z} = M(q_0) ,\\
	&
	&& \boldsymbol{y}_i \geq 0, \qquad i \in \{1,2 \}, \\
	&
	&& \boldsymbol{Z} \succeq 0.
	\end{aligned}
	\tag{SD}\label{prob: dual problem}
	\end{equation}
	Note that \eqref{prob: dual problem} is also the dual of \eqref{prob: original  problem}. 
	
	\begin{assumption}\label{asmp: Slater for SP}
		Problem \eqref{prob: semidefinite relaxation} satisfies Slater's condition, that is, there exists a symmetric positive definite $(n+1) \times (n+1)$-dimensional real matrix $X$ such that
		{
			\begin{subequations}
				\begin{align}
				M(q_i) \bigcdot X & <0, \qquad i \in \{1,2 \},\\
				I_{00} \bigcdot X & = 1.
				\end{align}
		\end{subequations}}
	\end{assumption}
	
	\begin{remark}
		\Cref{asmp: Slater for SP} holds when Slater's condition holds for \eqref{prob: very original problem} \cite{ai2009strong} (and, hence, for \eqref{prob: original  problem}), that is, there exists an $n$-dimensional vector $z$ such that
		{
			\begin{equation}\label{eq: Slater for SP}
			z^T Q_i z + 2 b_i^T z + c_i < 0, \qquad i \in \{1,2\}.
			\end{equation}}
	\end{remark}
	
	\begin{assumption}\label{asmp: Slater for SD}
		Slater's condition holds for \eqref{prob: dual problem}, that is, there exist a scalar $y_0$ and positive scalars $y_1$ and $y_2$ such that
		{
			\begin{equation}\label{eq: Slater for SD}
			{M(q_0)-y_0 I_{00} + y_1 M(q_1) + y_2 M(q_2) \succ 0.}
			\end{equation}}
	\end{assumption}
	
	\begin{remark}
		For problem data $M(q_0), M(q_1),$ and $M(q_2)$, one can numerically check whether \Cref{asmp: Slater for SP} and \cref{asmp: Slater for SD} are met by solving the feasibility problem of \eqref{prob: semidefinite relaxation} and \eqref{prob: dual problem}, respectively, using an SDP solver.
	\end{remark}
	
	\begin{remark}\label{remark: equivalent condition of Slater's for SD}
		The inequality \eqref{eq: Slater for SD} holds, by Schur complement, if and only if there exist a scalar $y_0$ and positive scalars $y_1$ and $y_2$ such that
		{
			\begin{subequations}\label{eq: Slater's condition for SD}
				\begin{align}
				Q_0 + y_1 Q_1 + y_2 Q_2 & \succ 0,\label{eq: Slater's condition for SD 1}\\
				\xi(y_1,y_2)^T (Q_0 + y_1 Q_1 + y_2 Q_2)^{-1} \xi(y_1,y_2) & < -y_0+y_1 c_1 + y_2 c_2,
				\end{align}
		\end{subequations}}
		where $\xi(y_1,y_2) \defeq b_0 + y_1 b_1 + y_2 b_2$.
	\end{remark} 
	
	\begin{remark}\label{remark: Ai and Zhang's special case of assumption}
		By \cite[Proposition~2.1]{ye2003new}, \Cref{asmp: Slater for SD} holds if the objective function of \eqref{prob: very original problem} is strictly convex (that is, $Q_0$ is positive definite), or at least one of the constraints of \eqref{prob: very original problem} is elliptical (that is, $Q_i \succ 0$ and $b_i^T Q_i^{-1} b_i -c_i >0$ for $i$ being $1$ or $2$, or both). The latter condition, which implies \eqref{eq: Slater for SD} by \Cref{remark: equivalent condition of Slater's for SD}, is required by Ai~and~Zhang in \cite{ai2009strong} which ensures that Slater's condition holds for \eqref{prob: dual problem}.
	\end{remark}

	\begin{remark}\label{remark: tight relaxation and strong duality}
		\Cref{asmp: Slater for SP} and \Cref{asmp: Slater for SD} together imply that both \eqref{prob: semidefinite relaxation} and \eqref{prob: dual problem} have attainable optimal solutions that yield an identical optimal value. This, in turn, implies that \eqref{prob: semidefinite relaxation} is a tight relaxation of \eqref{prob: original  problem} (that is, the optimal value of \eqref{prob: semidefinite relaxation} is identical to that of \eqref{prob: original  problem}) if and only if strong duality holds for \eqref{prob: original  problem}. This observation will be relevant later on when we establish our optimality gap test.
	\end{remark}
	
	We denote optimal solutions of \eqref{prob: original  problem}, \eqref{prob: semidefinite relaxation}, and \eqref{prob: dual problem}, respectively, by $x^*, \hat{X},$ and $(\hat{Z},\hat{y}_0,\hat{y}_1,\hat{y}_2)$, and their optimal values, respectively, by $\optval{\eqref{prob: original  problem}}$, $\optval{\eqref{prob: semidefinite relaxation}}$, and $\optval{\eqref{prob: dual problem}}$. Note that a primal-dual feasible pair, $X$ and $(Z,y_0,y_1,y_2)$, is optimal if and only if it satisfies the complementary conditions:
	\begin{subequations}\label{eq: complementary condition}
		\begin{align}
		XZ & = 0_{(n+1)\times (n+1)}, \label{eq: complementary condition-XZ} \\ \label{eq: complementary condition-q1}
		y_1 M(q_1) \bigcdot X & = 0,\\ \label{eq: complementary condition-q2}
		y_2 M(q_2) \bigcdot X & = 0.
		\end{align}
	\end{subequations}

	\textit{Property~I}, which is defined below (\Cref{def: Wenbao's Property I}), is the key to the necessary and sufficient conditions for an optimality gap (or, equivalently, for a duality gap between \eqref{prob: original  problem} and \eqref{prob: dual problem}) stated in \Cref{thm: Wenbao's theorem} for when $Q_1$ is positive definite.
	{\SecondRevisionColor
		\begin{definition}[\mbox{\cite[Definition~4.1]{ai2009strong}}]\label{def: Wenbao's Property I}
			For $\hat{X}$ and $(\hat{Z},\hat{y}_0,\hat{y}_1,\hat{y}_2)$, a given pair of optimal solutions for \eqref{prob: semidefinite relaxation} and \eqref{prob: dual problem}, respectively, we say that this pair has \textit{Property~I} if all the conditions~(\textit{I.1})--(\textit{I.4}) specified below hold:
			\begin{itemize}
				\item[(\textit{I.1})] $\hat{y}_1\hat{y}_2 \neq 0$;
				\item[(\textit{I.2})] $\rank(\hat{Z}) = n-1$;
				\item[(\textit{I.3})] $\rank(\hat{X}) = 2$;
				\item[(\textit{I.4})] There is a rank-one decomposition of $\hat{X}$, of the form $\hat{X} = \hat{x}_1\hat{x}_1^T + \hat{x}_2\hat{x}_2^T$, for which the following hold: 
				\begin{align} 
				M(q_1) \bigcdot \hat{x}_i\hat{x}_i^T & = 0, \qquad i \in \{1,2\},\label{eq: equality on the first constraint}\\
				(M(q_2) \bigcdot \hat{x}_1\hat{x}_1^T)(M(q_2) \bigcdot \hat{x}_2\hat{x}_2^T)& <0.\label{eq: inequality on the second constraint}
				\end{align}
			\end{itemize}
	\end{definition} }
	
	\begin{theorem}[\mbox{\cite[Theorem~4.2]{ai2009strong}}]\label{thm: Wenbao's theorem}
		Consider \eqref{prob: original  problem} where Slater's condition is satisfied and $Q_1$ is positive definite. {\ThirdRevisionColor Denote with $\mathfrak{P}$ the set of all pairs $\hat{X}'$ and $(\hat{Z}',\hat{y}_0',\hat{y}_1',\hat{y}_2')$ of optimal solutions for the semidefinite relaxation \eqref{prob: semidefinite relaxation} and the dual problem \eqref{prob: dual problem}, respectively. For any given pair $\hat{X}$ and $(\hat{Z},\hat{y}_0,\hat{y}_1,\hat{y}_2)$ arbitrarily selected from $\mathfrak{P}$, } $\optval{\eqref{prob: semidefinite relaxation}}< \optval{\eqref{prob: original  problem}}$ holds if and only if the pair satisfies \textit{Property~I}.
	\end{theorem}

	\section{Main results}\label{chap: main results}
	We start with modifying \textit{Property~I} by adding an extra constraint and naming the resulting conditions \textit{Property~I$^+$} as follows:
	{\SecondRevisionColor
		\begin{definition}
			\label{def:PropertyI+}
			
			For $\hat{X}$ and $(\hat{Z},\hat{y}_0,\hat{y}_1,\hat{y}_2)$, a given pair of optimal solutions for \eqref{prob: semidefinite relaxation} and \eqref{prob: dual problem}, respectively, we say that this pair has \textit{Property~I$^+$} if all the conditions~(\textit{I.1})--(\textit{I.3}) and~(\textit{I$^+$.4}) specified below hold:
			\begin{itemize}
				\item[(\textit{I.1})] $\hat{y}_1\hat{y}_2 \neq 0$;
				\item[(\textit{I.2})] $\rank(\hat{Z}) = n-1$;
				\item[(\textit{I.3})] $\rank(\hat{X}) = 2$;
				\item[(\textit{I$^+$.4})] There is a rank-one decomposition of $\hat{X}$, of the form $\hat{X} = \hat{x}_1\hat{x}_1^T + \hat{x}_2\hat{x}_2^T$, for which not only~(\ref{eq: equality on the first constraint}) and~(\ref{eq: inequality on the second constraint}) are satisfied, but the following inequality also holds:
				\begin{equation} \label{eq:I+4Cond}
				M(q_1) \bigcdot \hat{x}_1\hat{x}_2^T  \neq 0.
				\end{equation}
			\end{itemize}
		\end{definition}
		
		Notice that since the conditions (\textit{I.1})--(\textit{I.3}) are the same in the definitions of \textit{Properties~I} and~\textit{I$^+$}, the fact that (\textit{I$^+$.4}) is more stringent than (\textit{I.4}) is what distinguishes the properties.}
	As we shall see in the following theorem, \textit{Property~I$^+$} is the key to the necessary and sufficient condition for the optimality gap (or, equivalently, for the duality gap between \eqref{prob: original  problem} and \eqref{prob: dual problem}), when the positive definiteness of $Q_1$ is not assumed.

	\begin{theorem} \label{thm: necessary and sufficient condition for strong duality}
		Consider \eqref{prob: original  problem} and let \Cref{asmp: Slater for SP} and \Cref{asmp: Slater for SD} hold. {\ThirdRevisionColor Denote with $\mathfrak{P}$ the set of all pairs $\hat{X}'$ and $(\hat{Z}',\hat{y}_0',\hat{y}_1',\hat{y}_2')$ of optimal solutions for the semidefinite relaxation \eqref{prob: semidefinite relaxation} and the dual problem \eqref{prob: dual problem}, respectively. For any given pair $\hat{X}$ and $(\hat{Z},\hat{y}_0,\hat{y}_1,\hat{y}_2)$ arbitrarily selected from $\mathfrak{P}$, } ${\optval{\eqref{prob: semidefinite relaxation}}< \optval{\eqref{prob: original  problem}}}$ holds if and only if the pair satisfies \textit{Property~I$^+$}.
	\end{theorem}
	
	\begin{proof}
		See \cref{sec: proof}.
	\end{proof}
	
	{\ThirdRevisionColor
		\begin{remark} Notice that Theorem~\ref{thm: necessary and sufficient condition for strong duality}, and its proof given in~\cref{sec: proof}, hold irrespective of the choice of the pair $\hat{X}$ and $(\hat{Z},\hat{y}_0,\hat{y}_1,\hat{y}_2)$ so long as it is in $\mathfrak{P}$. Hence, one immediate conclusion from Theorem~\ref{thm: necessary and sufficient condition for strong duality} is that either all pairs in $\mathfrak{P}$ satisfy \textit{Property~I$^+$} or all pairs in $\mathfrak{P}$ violate \textit{Property~I$^+$}. This is particularly relevant for applications because in order to use the theorem to determine whether there is an optimality gap, it suffices to check if one (arbitrarily chosen) pair in $\mathfrak{P}$ satisfies \textit{Property~I$^+$}. Moreover, it will follow from our discussion in section~\ref{chap: test}, where we describe an algorithm (\Cref{algo: strong duality}) to test \textit{Property~I$^+$}, that if (\textit{I$^+$.1})--(\textit{I$^+$.4}) are checked in the order and manner specified in the algorithm then the validity of (\textit{I$^+$.4}) can be determined by checking only one rank-one decomposition of $\hat{X}$ arbitrarily chosen among those satisfying~(\ref{eq: equality on the first constraint}). Remark~\ref{rem:AboutAlgo4.1} gives a justification of this fact by leveraging our proof of Theorem~\ref{thm: necessary and sufficient condition for strong duality} in~\cref{sec: proof}.
		\end{remark} 
		
		Notice that the test in \Cref{thm: necessary and sufficient condition for strong duality} is tractable because it involves the solution of an SDP, and its dual, followed by a rank-one decomposition, both of which run in polynomial time~\cite{ye2003new}.} {\firstRevisionColor Furthermore,
		the necessity part of the proof (\cref{sec: proof}) of \Cref{thm: necessary and sufficient condition for strong duality} is constructive and useful in its own right because it embeds the description of a method to \underline{construct a solution for~\eqref{prob: original  problem}} from that of~\eqref{prob: semidefinite relaxation} when \textit{Property~I$^+$} is violated. }
	
	Besides the numerical results that are going to be presented in the next section, \Cref{thm: necessary and sufficient condition for strong duality} was used in~\cite[Theorem~2]{chengmartins2018reaching} to propose conditions under which there is no optimality gap for a control-oriented QC2QP problem.
	
	
	
	{\firstRevisionColor
		
		The following remark is relevant due to the intriguing fact that, although $M(q_1)$ and $M(q_2)$ play identical ``roles" in  \eqref{prob: original  problem} and \eqref{prob: semidefinite relaxation}, they enter Definitions~\ref{def: Wenbao's Property I} and~\ref{def:PropertyI+} in different ways. 
		\begin{remark} \label{rem:symmetry}
			One can readily verify by inspection that interchanging $M(q_1)$ and $M(q_2)$ in \eqref{prob: original  problem} and~\eqref{prob: semidefinite relaxation} would not modify the underlying problem and, in particular, would not alter whether $\optval{\eqref{prob: semidefinite relaxation}}< \optval{\eqref{prob: original  problem}}$ holds. Similarly, \Cref{thm: necessary and sufficient condition for strong duality} would remain valid had we chosen to  interchange $M(q_1)$ and $M(q_2)$ in Definitions~\ref{def: Wenbao's Property I} and~\ref{def:PropertyI+}. 
			
		\end{remark}
	}
	\interfootnotelinepenalty=10000 
	{\firstRevisionColor 
		We now proceed to remark on the differences between the proofs given in~\cref{sec: proof} and \cite{ai2009strong} for Theorems~\ref{thm: necessary and sufficient condition for strong duality} and~\ref{thm: Wenbao's theorem}, respectively.
		
		\begin{remark}
			We start by observing that, in~\cref{sec: proof}, we prove the sufficient condition of~Theorem~\ref{thm: necessary and sufficient condition for strong duality} by exploring the fact that $\hat{X}$ is the unique solution of \eqref{prob: semidefinite relaxation} when \textit{Property~I${}^+$} holds, whereas the proof of the sufficient condition of~\Cref{thm: Wenbao's theorem} is ascertained in~\cite{ai2009strong} using contradiction.
			The proofs of the necessary condition in~\cite{ai2009strong} and~\cref{sec: proof} construct a rank-one solution to~\eqref{prob: semidefinite relaxation} for the possible cases in which~\textit{Property~I} and \textit{Property~I${}^+$} fail, respectively. Although the four cases enumerated in the former are also present in the latter, there are differences in the analysis. Specifically, in~\cref{sec: proof} we need to invoke \Cref{asmp: Slater for SD} in cases 1, 2, and 4 to show the existence of a nonzero solution, while the argument in~\cite{ai2009strong} uses the positive definiteness\footnote{\firstRevisionColor Recall that in our article $Q_0$, $Q_1$, and $Q_2$ are only required to satisfy \Cref{asmp: Slater for SP} and \Cref{asmp: Slater for SD}. It may be the case that none is positive definite.} of $Q_1$. In addition, we need to consider the fifth case in~\cref{sec: proof} to account for when \textit{Property~I${}^+$} fails and \textit{Property~I} holds. 
		\end{remark}
	}

	\section{Testing \textit{Property~I$^+$} numerically}\label{chap: test}
	
	\Cref{thm: necessary and sufficient condition for strong duality} enables a simply verifiable optimality gap test for a semidefinite relaxation of a QC2QP. This test only requires to solve one SDP with its dual and conducting a rank-one decomposition, both of which run in polynomial time. However, in general, SDP solvers (for example, SDPT3~\cite{toh1999sdpt3} and SeDuMi~\cite{sturm1999using}) give approximate, rather than exact, solutions within certain tolerance. Hence, it is useful to establish an optimality gap test utilizing {\textit{Property~I$^{+}$}} in an approximation sense. The following procedures refer to the purified $(\epsilon_1,\epsilon_2)$-approximation in \cite{ai2009strong}:
	
	Let $\hat{X}$ and $(\hat{Z},\hat{y}_0,\hat{y}_1,\hat{y}_2)$ denote a pair of numerical solutions of \eqref{prob: semidefinite relaxation} and its dual \eqref{prob: dual problem}, respectively, solved by an SDP solver whose tolerance is $\epsilon_1 > 0$. Let $\epsilon_2>0$ be the tolerance for purification. First, conduct an eigendecomposition of $\hat{X}$ and $\hat{Z}$, that is,
	\begin{gather}
	\hat{X} = P_1^T \Lambda_1 P_1,\\
	\hat{Z} = P_2^T \Lambda_2 P_2,
	\end{gather}
	where $P_1$ and $P_2$ are $(n+1)\times (n+1)$-dimensional orthogonal matrices, and $\Lambda_1$ and $\Lambda_2$ are $(n+1) \times (n+1)$-dimensional diagonal matrices. Let $\Lambda_i \defeq \text{diag}(\lambda_{i1},\dots,\lambda_{i(n+1)}),$ and let 
	\begin{equation}
	\lambda_{ij}^* \defeq \left\{ 
	\begin{aligned}
	\lambda_{ij},\quad & \lambda_{ij} \geq \epsilon_2\\
	0,\quad & \lambda_{ij} < \epsilon_2
	\end{aligned}
	\right., \quad i \in \{1,2 \}, \quad j \in \{1,2,\dots,n+1\}.
	\end{equation}
	Define the purified solutions by
	\begin{align}
	X^* & \defeq P_1^T \text{diag}(\lambda_{11}^*,\dots,\lambda_{1(n+1)}^*)P_1,\\
	Z^* & \defeq P_2^T \text{diag}(\lambda_{21}^*,\dots,\lambda_{2(n+1)}^*)P_2,\\
	y_i^* &\defeq \hat{y}_i, \qquad i \in \{0,1,2\}.
	\end{align}
	The above step essentially purifies the rank of $\hat{X}$ and $\hat{Z}$ so that
	\begin{align}
	\rank(\hat{X},\epsilon_2) & = \rank(X^*,\epsilon_2),\\
	\rank(\hat{Z},\epsilon_2) & = \rank(Z^*,\epsilon_2).
	\end{align}
	
	\begin{remark}
		We use \textit{numerical $\epsilon$-rank}, as defined in \cite{hansen2005rank}, to determine the numerical rank of a matrix with tolerance $\epsilon$. Namely, the numerical rank of a matrix $A$ in $\mathbb{R}^{m \times n}$ with tolerance $\epsilon$, denoted by $\rank(A,\epsilon)$, is $r$, if the singular values, $\sigma_1 \geq \sigma_2 \geq \dots \geq \sigma_{\text{min}(m,n)} \geq 0$, of $A$ satisfy $\sigma_1 \geq \sigma_2 \geq \dots \geq \sigma_r > \epsilon > \sigma_{r+1} \geq \dots \geq \sigma_{\text{min}(m,n)}$. For more details, see \cite{golub2012matrix}~and~\cite{hansen2005rank}.
	\end{remark}
	
	Henceforth, we use $X^*$ and $(Z^*,y_0^*,y_1^*,y_2^*)$ to denote a pair of purified $(\epsilon_1,\epsilon_2)$-approximate optimal solutions of \eqref{prob: semidefinite relaxation} and \eqref{prob: dual problem}, respectively.
	
	\begin{definition}\label{def: property I in approximation}
		For $X^*$ and $(Z^*,y_0^*,y_1^*,y_2^*)$, a given pair of purified $(\epsilon_1,\epsilon_2)$-approximate optimal solutions of \eqref{prob: semidefinite relaxation} and \eqref{prob: dual problem}, respectively, we say that this pair has \textit{Property~I$^+(\epsilon_2)$} if:
		\begin{enumerate}[leftmargin=2cm]
			\item[(\textit{I$^+(\epsilon_2)$.1})] $y_1^* > \epsilon_2$ and $y_2^* > \epsilon_2$;
			\item[(\textit{I$^+(\epsilon_2)$.2})] $\rank(Z^*,\epsilon_2) = n-1$;
			\item[(\textit{I$^+(\epsilon_2)$.3})] $\rank(X^*,\epsilon_2) = 2$;
			\item[(\textit{I$^+(\epsilon_2)$.4})] There is a rank-one decomposition of $X^*$, of the form $X^*= x_1^*(x_1^*)^T + x_2^*(x_2^*)^T$, for which the following hold:
			\begin{align}
			| M(q_1) \bigcdot x_i^*(x_i^*)^T | & < \epsilon_2, \qquad i \in \{1,2\},\label{eq: epsilon inequality 0}\\
			\revisionColor \left ( M(q_2) \bigcdot x_1^*(x_1^*)^T \right) \left( M(q_2) \bigcdot x_2^*(x_2^*)^T \right) & \revisionColor <- \epsilon_2^2 ,\label{eq: epsilon inequality 1}\\
			|M(q_1) \bigcdot x_1^*(x_2^*)^T| &> \epsilon_2. \label{eq: epsilon inequality 3}
			\end{align}
		\end{enumerate}
	\end{definition}

	Now, we introduce a polynomial-time algorithm to test the optimality gap for the relaxation \eqref{prob: semidefinite relaxation}, with $\epsilon_1$-precision SDP solutions $\hat{X}$ and $(\hat{Z},\hat{y}_0,\hat{y}_1,\hat{y}_2)$ as well as \textit{Property~I$^+(\epsilon_2)$}.
	
	\begin{algorithm}[H]
		\caption{Optimality gap test}
		\label{algo: strong duality}
		\begin{algorithmic}[1]
			\REQUIRE $\epsilon_1,\epsilon_2,Q_0,Q_1,Q_2,q_0,q_1,q_2,c_1,c_2$
			\STATE{Solve \eqref{prob: semidefinite relaxation} and $\eqref{prob: dual problem}$ for $\hat{X}$ and $(\hat{Z},\hat{y}_0,\hat{y}_1,\hat{y}_2)$, respectively, using an SDP solver with $\epsilon_1$-precision.}
			\STATE{Compute the purified $(\epsilon_1,\epsilon_2)$-approximate optimal pair of solutions $X^*$ and $(Z^*,y^*_0,y^*_1,y^*_2)$.}
			{\SecondRevisionColor
				\STATE{Check the conditions of \textit{Property I$^+(\epsilon_2)$} in the following order:}
				\IF{(\textit{I$^+(\epsilon_2)$.1}), (\textit{I$^+(\epsilon_2)$.2}), or (\textit{I$^+(\epsilon_2)$.3}) is violated}
				\STATE{There is no optimality gap.}
				\ELSE
				\STATE{Using \cite[Lemma~2.2]{ye2003new}, conduct a rank-one decomposition of $X^*$ of the form $X^*= x_1^*(x_1^*)^T + x_2^*(x_2^*)^T$ satisfying~\eqref{eq: epsilon inequality 0}.}
				\IF{\revisionColor $x_1^*$ and $x_2^*$ violate \eqref{eq: epsilon inequality 1} 
				}
				\STATE{There is no optimality gap.}
				\ELSIF{$x_1^*$ and $x_2^*$ violate \eqref{eq: epsilon inequality 3}}
				\STATE{There is no optimality gap.}
				\ELSE
				\STATE{An optimality gap exists.}
				\ENDIF
				\ENDIF
			}
		\end{algorithmic}
	\end{algorithm}
	

	{\SecondRevisionColor
		\begin{remark}
			\Cref{algo: strong duality} can always find the gap when $\epsilon_2$ is chosen properly. The algorithm relies on \textit{Property~I$^+(\epsilon_2)$} instead of \textit{Property~I$^+$} because the latter is exact, which cannot handle the round-off errors in numerical computations. Hence, $\epsilon_2$ must be carefully chosen so that
			\begin{enumerate}
				\item The terms that are zero-valued in \textit{Property~I$^+$} will be in the $\epsilon_2$-neighborhood of zero subject to round-off errors. Likewise, the terms that are nonzero-valued will be out of the $\epsilon_2$-neighborhood of zero.
				\item Only the significant eigenvalues (those greater or equal to $\epsilon_2$) of $\hat{X}$ and $\hat{Z}$ are counted towards the numerical rank of each matrix.
			\end{enumerate}
		\end{remark}
	}
	
	By \Cref{remark: tight relaxation and strong duality}, the test described by \Cref{algo: strong duality} can also be applied to check whether a duality gap exists for \eqref{prob: original  problem}. Note that it could be difficult to check the existence of the duality gap using other methods, for example, solving for the global minimum of the primal problem. The primal problem is nonconvex, though smooth. Hence, to find the global minimum, a good initial point is necessary for the convergence of a gradient-based algorithm. However, to the best of the authors' knowledge, there is no efficient method to pick an initial point for convergence to the global minimum.
	
	{\ThirdRevisionColor
		\begin{remark}
			\label{rem:AboutAlgo4.1}
			We proceed to explain why in Steps~8~and~10 of \Cref{algo: strong duality} we test whether only one rank-one decomposition $(x_1^*,x_2^*)$ (arbitrarily selected among those complying with \eqref{eq: epsilon inequality 0}) violates \eqref{eq: epsilon inequality 1} 
			or \eqref{eq: epsilon inequality 3}. A clarification is warranted because, even when it holds, the test {\revisionColor in and of itself} does not preclude the existence of another rank-one decomposition satisfying \eqref{eq: epsilon inequality 0}--\eqref{eq: epsilon inequality 3}. In our justification, we will repurpose portions of the proof of \Cref{thm: necessary and sufficient condition for strong duality} in section~\ref{sec: proof}. Namely, {\revisionColor if $(x_1^*,x_2^*)$ violates \eqref{eq: epsilon inequality 1}
			} in Step~8 then we can use $(x_1^*,x_2^*)$ as the rank-one pair selected in Case~3 of the necessity proof in section~\ref{sec: proof} and follow the procedure therein to establish constructively the absence of an optimality gap. 
			{\revisionColor Conversely, if $(x_1^*,x_2^*)$ satisfies \eqref{eq: epsilon inequality 1}
				, which means that $(x_1^*,x_2^*)$ ``approximately'' belongs to the set $\mathbb{O}^+(\hat{X})$ specified in Remark~\ref{rem:setO} of section~\ref{sec: proof}, and if $(x_1^*,x_2^*)$ violates \eqref{eq: epsilon inequality 3} in Step~10 then we can use $(x_1^*,x_2^*)$ as the rank-one pair selected in Case~5 of the proof and establish the absence of the optimality gap.}
			Otherwise, if $(x_1^*,x_2^*)$ satisfies \eqref{eq: epsilon inequality 1} 
			and \eqref{eq: epsilon inequality 3} in Step~12 then the pair $X^*$ and $(Z^*,y^*_0,y^*_1,y^*_2)$ satisfies \textit{Property I$^+(\epsilon_2)$}, in which case we can follow the sufficiency proof in section~\ref{sec: proof} to show that there is an optimality gap.
		\end{remark}
	}

	
	We have implemented \Cref{algo: strong duality} in a MATLAB script which is available online~\cite{cheng2018GitHubCode}\footnote{ \firstRevisionColor The description and data of a numerical experiment that tests the proportion of the randomly generated feasible nonconvex QC2QP instances of which there is no optimality gap are also available online~\cite{cheng2018GitHubCode}.}. We invite the interested readers to use this script with their own QC2QP problem data.
	
	We provide numerical examples which contain two nonconvex QC2QPs, in order to illustrate the test described in \Cref{algo: strong duality}. The optimality gap does not exist in the first example but does in the second one. Note that \Cref{thm: Wenbao's theorem} cannot be applied to the optimality gap test because the constraints in both examples are nonconvex, which violates the assumption in~\cite{ai2009strong} that requires at least one of the constraints to be strictly convex. 
	
	To solve SDPs, we use CVX \cite{gb08}\cite{cvx} with solver SDPT3 \cite{toh1999sdpt3} and the default tolerance $\epsilon_1 = 1.49\times 10^{-8}$. We set the purification tolerance $\epsilon_2$ to $1 \times 10^{-5}$. 
	
	\subsection{Example: (there is no optimality gap)}
	\label{example:1}
	Consider the following data in~\eqref{prob: very original problem}:
	\begin{gather}
	Q_0 = \begin{bmatrix}
	2 & -4 \\
	-4 & -2
	\end{bmatrix},\quad Q_1 = \begin{bmatrix}
	4 & -5 \\
	-5 & 2
	\end{bmatrix},\quad Q_2 = \begin{bmatrix}
	0 & 2 \\
	2 & 2
	\end{bmatrix}, \\
	b_0 = \begin{bmatrix}
	0 \\ 0
	\end{bmatrix},\quad b_1 = \begin{bmatrix}
	2 \\ 0
	\end{bmatrix},\quad b_2 = \begin{bmatrix}
	0 \\ 5
	\end{bmatrix},\quad c_1 = -1,\quad c_2 = -4. 
	\end{gather}
	Both the objective function and constraints are hyperbolic and nonconvex. \Cref{asmp: Slater for SP} and \Cref{asmp: Slater for SD} are verified to hold in this example. Solving \eqref{prob: semidefinite relaxation} and \eqref{prob: dual problem}, we obtain the purified $(\epsilon_1,\epsilon_2)$-approximate optimal solutions: 
	\begin{align}
	X^* & \approx \begin{bmatrix}
	1.0000000 & -0.7547192 & -3.9916123 \\
	-0.7547192 &  0.5696011 & 3.0125464 \\
	-3.9916123  & 3.0125464 & 15.9329684
	\end{bmatrix},\\
	Z^* & \approx \begin{bmatrix}
	45.5612496 &  0.3855596 & 11.3413460 \\
	0.3855596 &  2.7711193 & -0.4273607 \\
	11.3413460 & -0.4273607 &  2.9220980
	\end{bmatrix},\\
	y_0^* & \approx -54.8271062, \quad  y_1^* \approx 0.1927798, \quad  y_2^* \approx 2.2682692,
	\end{align}
	where $\rank(X^*,\epsilon_2) = 1$, $\rank(Z^*,\epsilon_2) = 2$, and a rank-one decomposition yields ${X^* = x^*(x^*)^T}$ such that 
	\begin{gather}
	x^* \approx \begin{bmatrix}
	-1.0000000 &
	0.7547192 &
	3.9916123
	\end{bmatrix}^T.
	\end{gather}
	
	It can be verified that 
	\begin{gather}
	|M(q_1) \bigcdot x^* (x^*)^T| < \epsilon_2, \\
	|M(q_2) \bigcdot x^* (x^*)^T| < \epsilon_2.
	\end{gather}
	
	Denote the normalized solutions by $\hat{z}$, where $\hat{z} = z^*/t^*$ follows the partition of $x^*$ in the form of \eqref{eq: partition of x}. Thus, 
	\begin{equation}
	\hat{z} \approx \begin{bmatrix}
	-0.7547192 &
	-3.9916123
	\end{bmatrix}^T
	\end{equation}
	is marked in \Cref{fig: strong duality}.
	
	Since \textit{Property~I$^+(\epsilon_2)$} is violated, we claim that there is no optimality gap. This can be verified since the primal optimal solution
	\begin{equation}
	z^+ \approx \begin{bmatrix}
	-0.7547192 &
	-3.9916123
	\end{bmatrix}^T
	\end{equation}
	coincides with $\hat{z}$. The globally optimal value is $-54.8271061$, which is identical to that of \eqref{prob: semidefinite relaxation}.
	
	\begin{figure}[tbhp]
		\centering
		\includegraphics[width = \textwidth]{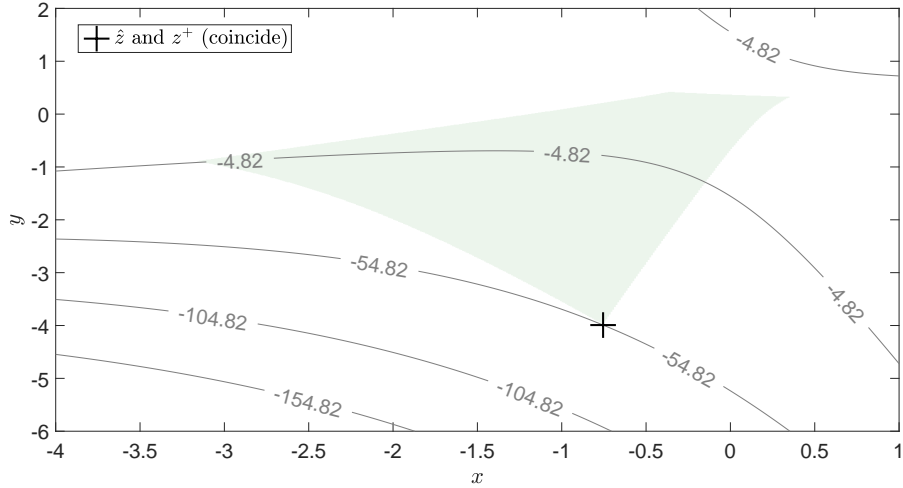}	
		\caption{Illustration of the numerical example in which there is no optimality gap. The contour plot represents the objective function, whereas the filled region represents the feasible set. The primal optimum $z^+$ coincides with the normalized solution $\hat{z}$ from \eqref{prob: semidefinite relaxation}.}
		\label{fig: strong duality}
	\end{figure}
	
	\subsection{Example: (there is an optimality gap)}
	\label{example:2}
	Consider the following data in \eqref{prob: very original problem}:
	\begin{gather}
	Q_0 = \begin{bmatrix}
	-1 & -2 \\
	-2 & 1
	\end{bmatrix},\quad Q_1 = \begin{bmatrix}
	3 & 1 \\
	1 & -2
	\end{bmatrix},\quad Q_2 = \begin{bmatrix}
	4 & 5 \\
	5 & 1
	\end{bmatrix}, \\
	b_0 = \begin{bmatrix}
	-2 \\ 0
	\end{bmatrix},\quad b_1 = \begin{bmatrix}
	3 \\ 2
	\end{bmatrix},\quad b_2 = \begin{bmatrix}
	-1 \\ 5
	\end{bmatrix},\quad c_1 = -2,\quad c_2 = 4. 
	\end{gather}
	Both the objective function and constraints are hyperbolic and nonconvex. \Cref{asmp: Slater for SP} and \Cref{asmp: Slater for SD} are verified to hold in this example. Solving \eqref{prob: semidefinite relaxation} and \eqref{prob: dual problem}, we obtain the purified $(\epsilon_1,\epsilon_2)$-approximate optimal solutions:
	\begin{align}
	X^* & \approx \begin{bmatrix}
	1.0000000 &  0.9982700 & -1.2814553 \\
	0.9982700 & 2.2688396 &  -0.0999477 \\
	-1.2814553 & -0.0999477 &  2.7352111
	\end{bmatrix},\\
	Z^* & \approx \begin{bmatrix}
	3.4958344 &  -1.4683240 &  1.5841752\\
	-1.4683240 &  0.6167270 & -0.6653869\\
	1.5841752 & -0.6653869 &  0.7178861
	\end{bmatrix},\\
	y_0^* & \approx -3.1269177, \quad  y_1^* \approx 0.2495621, \quad  y_2^* \approx 0.2170102,
	\end{align}
	where $\rank(X^*,\epsilon_2) = 2$, $\rank(Z^*,\epsilon_2) = 1$, and a rank-one decomposition yields $X^* = x_1^*(x_1^*)^T + x_2^*(x_2^*)^T$ such that 
	\begin{align}
	x_1^* & \approx \begin{bmatrix}
	0.1712233 & -0.9403767 & -1.2494816
	\end{bmatrix}^T,\\
	x_2^* & \approx \begin{bmatrix}
	0.9852322 & 1.1766611 & -1.0835160
	\end{bmatrix}^T.
	\end{align}
	
	It can be verified that 
	\begin{align}
	| M(q_1) \bigcdot x_i^* (x_i^*)^T| &< \epsilon_2, \qquad i \in \{1,2\},\\
	M(q_2) \bigcdot x_1^* (x_1^*)^T &> \epsilon_2, \\
	\ M(q_2) \bigcdot x_2^* (x_2^*)^T &< -\epsilon_2,\\
	|M(q_1) \bigcdot x_1^* (x_2^*)^T| &> \epsilon_2.
	\end{align}
	
	Denote the normalized solutions by $\hat{z}_1$ and $\hat{z}_2$, where $\hat{z}_1 = z_1^*/t_1^*$ and $\hat{z}_2 = z_2^*/t_2^*$ follow the partition of $x_1^*$ and $x_2^*$, respectively, in the form of \eqref{eq: partition of x}. Thus, 
	\begin{align}
	\hat{z}_1 & \approx \begin{bmatrix}
	-5.4921056 & -7.2973787
	\end{bmatrix}^T, \\
	\hat{z}_2 & \approx \begin{bmatrix}
	1.1942982 &	-1.0997569
	\end{bmatrix}^T,
	\end{align}
	are marked in \Cref{fig: nostrongduality}.
	
	Since \textit{Property~I$^+(\epsilon_2)$} is met, we claim that there is an optimality gap. This can be verified since the primal optimal solution 
	\begin{equation}
	z^+ \approx \begin{bmatrix}
	0.5251114 &	-0.3446140
	\end{bmatrix}^T,
	\end{equation}
	which is marked in \Cref{fig: nostrongduality}, yields the globally optimal value $-1.5335857$, whereas the optimal value of \eqref{prob: semidefinite relaxation} is $-3.1269177$.

	\begin{figure}[tbhp]
		\centering
		\includegraphics[width =  \textwidth]{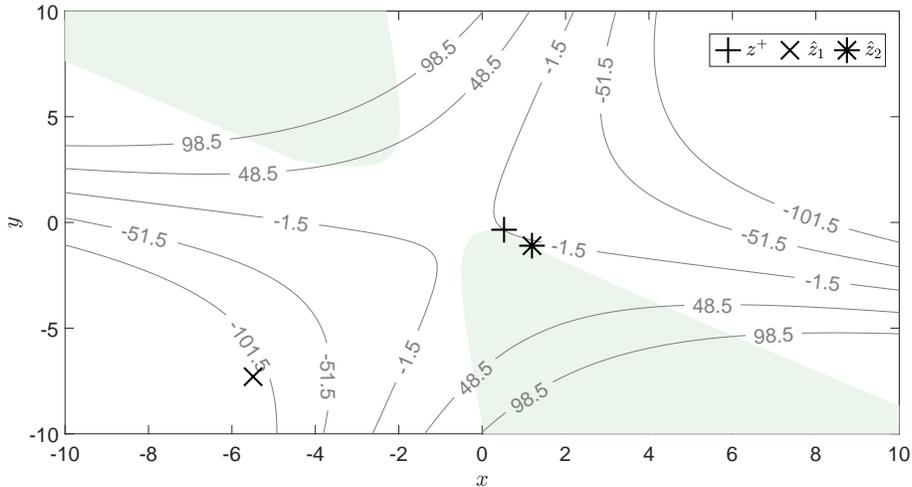}	
		\caption{Illustration of the numerical example in which there is an optimality gap. The contour plot represents the objective function, whereas the filled region represents the feasible set.}
		\label{fig: nostrongduality}
	\end{figure}

	\section{On why \Cref{thm: necessary and sufficient condition for strong duality} is an extension of \Cref{thm: Wenbao's theorem}} \label{chap: comparison}
	
	The following facts support the claim that~\Cref{thm: necessary and sufficient condition for strong duality} is an extension of \Cref{thm: Wenbao's theorem}: 
	\begin{itemize}
		\item The assumptions required for the validity of~\Cref{thm: necessary and sufficient condition for strong duality} are weaker than those needed in~\Cref{thm: Wenbao's theorem} (see~\Cref{remark: Ai and Zhang's special case of assumption} for more details). 
		
		\item \textit{Property~I$^+$}, which determines the test in the necessary and sufficient condition of~\Cref{thm: necessary and sufficient condition for strong duality}, involves an extra condition when compared with \textit{Property~I} used in~\Cref{thm: Wenbao's theorem}. Furthermore, as the following proposition establishes, \textit{Property~I$^+$} and~\textit{Property~I} are equivalent when $Q_1 \succ 0$, as assumed in~\Cref{thm: Wenbao's theorem}.
	\end{itemize}
	
	In short, if the assumptions of~\Cref{thm: Wenbao's theorem} are satisfied then those of~\Cref{thm: necessary and sufficient condition for strong duality} are also valid and the necessary and sufficient test is equivalent. However, the opposite does not hold, as the parameters defining the examples in~\cref{example:1} and~\cref{example:2} illustrate.
	
	
	\begin{proposition}\label{prop: restate I+'s redundancy in Wenbao's result}
		Consider \eqref{prob: original  problem} where Slater's condition holds and $Q_1 \succ 0$. Let $\hat{X}$ and $(\hat{Z},\hat{y}_0,\hat{y}_1,\hat{y}_2)$ denote a pair of optimal solutions for \eqref{prob: semidefinite relaxation} and \eqref{prob: dual problem}, respectively. If $\rank(\hat{X}) = 2$ and $\hat{y}_1 \neq 0$, then there exists a rank-one decomposition of $\hat{X}$, of the form $\hat{X} = \hat{x}_1\hat{x}_1^T + \hat{x}_2 \hat{x}_2^T$, which satisfies $M(q_1) \bigcdot \hat{x}_1\hat{x}_1^T=M(q_1) \bigcdot \hat{x}_2\hat{x}_2^T=0$ and $M(q_1) \bigcdot \hat{x}_1\hat{x}_2^T \neq 0$.
	\end{proposition}
	
	\begin{proof}
		By complimentary slackness \eqref{eq: complementary condition-q1}, $\hat{y}_1 \neq 0$ implies $M(q_1) \bigcdot \hat{X} = 0$. By \cite[Lemma~2.2]{ye2003new}, we can always obtain a rank-one decomposition $\hat{X} = \hat{x}_1 \hat{x}_1^T + \hat{x}_2 \hat{x}_2^T$ that satisfies 
		\begin{equation}\label{eq: intermediate equality}
		M(q_1) \bigcdot \hat{x}_1 \hat{x}_1^T = M(q_1) \bigcdot \hat{x}_2 \hat{x}_2^T = 0.	
		\end{equation}
		
		We proceed to show that the pair $(\hat{x}_1,\hat{x}_2)$ obtained from a rank-one decomposition of $\hat{X}$ that satisfies \eqref{eq: intermediate equality} also yields $M(q_1) \bigcdot \hat{x}_1 \hat{x}_2^T \neq 0$. We can conduct a change of coordinates $\tilde{x} = Rx$ to make
		\begin{equation}\label{eq: M(q_1) after linear transformation}
		M(q_1) = \begin{bmatrix}
		-1 & 0_{1 \times n} \\
		0_{n \times 1} & I_n
		\end{bmatrix}
		\end{equation}
		because
		\begin{eqnarray}
		\tilde{q}_1(\tilde{x}) & \defeq & \tilde{x}^T M(q_1) \tilde{x} \nonumber\\	
		&=& \begin{bmatrix}
		\tilde{t} \\
		\tilde{z}
		\end{bmatrix}^T \begin{bmatrix}
		-1 & 0_{1 \times n}\\
		0_{n \times 1} & I_n
		\end{bmatrix} \begin{bmatrix}
		\tilde{t}\\
		\tilde{z}
		\end{bmatrix} \nonumber\\
		& = &\begin{bmatrix}
		t\\
		z
		\end{bmatrix}^T  R^T \begin{bmatrix}
		-1 & 0_{1 \times n}\\
		0_{n \times 1} & I_n
		\end{bmatrix} R \begin{bmatrix}
		t \\
		z
		\end{bmatrix} \nonumber \\
		& = & z^T Q_1 z + 2 t b_1^T Q_1z + t^2c_1 = q_1(x),
		\end{eqnarray}
		where
		\begin{equation}
		R \defeq \begin{bmatrix}
		(b_1^T Q_1^{-1} b_1-c_1)^{1/2} & 0_{1 \times n} \\
		Q_1^{-1/2}b_1 & Q_1^{1/2}
		\end{bmatrix},    
		\end{equation}
		and $\bigl[ \begin{smallmatrix}
		{t} \\ {z}
		\end{smallmatrix} \bigr]$ and $\bigl[ \begin{smallmatrix}
		\tilde{t} \\ \tilde{z}
		\end{smallmatrix} \bigr]$ are partitions of $x$ and $\tilde{x}$, respectively, in the form of \eqref{eq: partition of x}. 
		
		By contradiction, assume $M(q_1) \bigcdot \hat{x}_1\hat{x}_2^T = 0$. Hence, adopting the partition of $\hat{x}_1$ and $\hat{x}_2$ in the form of \eqref{eq: partition of x} and using the $M(q_1)$ in \eqref{eq: M(q_1) after linear transformation}, we have 
		\begin{equation}
		\left\{
		\begin{array}{ll}
		M(q_1) \bigcdot \hat{x}_1\hat{x}_1^T = 0,\\
		M(q_1) \bigcdot \hat{x}_2\hat{x}_2^T = 0,\\
		M(q_1) \bigcdot \hat{x}_1\hat{x}_2^T = 0,
		\end{array}
		\right. \Rightarrow 
		\left\{
		\begin{array}{ll}
		\hat{z}_1^T \hat{z}_1 = \hat{t}_1^2,\\
		\hat{z}_2^T \hat{z}_2 = \hat{t}_2^2,\\
		\hat{z}_1^T \hat{z}_2 = \hat{t}_1\hat{t}_2,
		\end{array}
		\right. \Rightarrow (\hat{z}_1^T \hat{z}_1)(\hat{z}_2^T \hat{z}_2) = (\hat{z}_1^T \hat{z}_2)^2.
		\end{equation}
		By Cauchy--Schwarz inequality, the last equality implies $\hat{z}_1$ and $\hat{z}_2$ are linearly dependent, and so are $\hat{x}_1$ and $\hat{x}_2$, which contradicts the fact that $\rank(\hat{X}) = 2$ and $\hat{X} = \hat{x}_1\hat{x}_1^T + \hat{x}_2 \hat{x}_2^T$.
	\end{proof}
	
	\section{{\firstRevisionColor Conclusions and future directions}\label{sec: conclusions and future work}}
	{\firstRevisionColor In this article, we establish in~\Cref{thm: necessary and sufficient condition for strong duality} that a general QC2QP has an optimality gap if and only if it satisfies a condition we denote as \textit{Property~I${}^+$}, on account of the earlier test based on the so-called \textit{Property~I}~\cite{ai2009strong}. According to our result, there is no optimality gap in the desirable case in which \textit{Property~I${}^+$} is violated, and a solution of the QC2QP can be obtained from a rank-one decomposition of a solution of a semidefinite relaxation. Our result generalizes the state of the art stated in~\Cref{thm: Wenbao's theorem} (reproduced from~\cite[Theorem~4.2]{ai2009strong}), which uses \textit{Property~I} to establish a similar test subject to the additional requirement that at least one of the constraints is strictly convex. In contrast to \Cref{thm: Wenbao's theorem}, \Cref{thm: necessary and sufficient condition for strong duality} remains valid when none of the inequalities of the QC2QP is strictly convex, but they are equivalent otherwise. 
	}
	
	{\firstRevisionColor Future directions include tightening and eliminating the optimality gap of a QC2QP that is tested affirmative of \textit{Property~I${}^+$}. Here, one could possibly leverage existing methods that tighten or eliminate the gap for special cases of QC2QP such as, for example, those based on lifting~\cite{Yang2016Twovariable}, adding SOC constraints~\cite{yuan2017new}, or Lagrange multipliers~\cite{Sakaue2016Solving}. Another interesting possibility is to establish an $\epsilon$-version of \Cref{thm: necessary and sufficient condition for strong duality} which would bound the optimality gap by a function of $\epsilon$, when \textit{Property~I${}^+$} is violated by $\epsilon$.
	}
	
	\section{Appendix: Proof of~\Cref{thm: necessary and sufficient condition for strong duality}}\label{sec: proof}
	In order to prove \Cref{thm: necessary and sufficient condition for strong duality}, we shall use the following lemma whose proof follows from that of the sufficient condition of \cite[Theorem~2.4]{lemon2016low}. {\firstRevisionColor As is evident from the analysis in \cite[p.~19]{lemon2016low}, this result does not rely on the assumption \cite[(2.1)]{lemon2016low} required in other portions of \cite[Chapter~2]{lemon2016low}.
		
		\begin{lemma} \label{lemma: SDP unique solution}
			Let $X = VV^T$ be a solution of \eqref{prob: semidefinite relaxation}, where $V$ is a matrix in $\mathbb{R}^{n \times r}$ and $r = \rank(X)$. Define a linear mapping $\mathcal{A}_V:\mathcal{S}^r \rightarrow \mathbb{R}^3$ as
			\begin{equation}\label{eq: range space A V}
			\mathcal{A}_V(\Delta) = \begin{bmatrix}
			(V^T M(q_1)V) \bigcdot \Delta \\
			(V^T M(q_2)V) \bigcdot \Delta \\
			(V^T I_{00} V) \bigcdot \Delta \\
			\end{bmatrix},
			\end{equation} and consider the following conditions:
			\begin{itemize}
				\item[] {\bf (Lemma~\ref{lemma: SDP unique solution}-C1)} All solutions $\tilde{X}$ of \eqref{prob: semidefinite relaxation} have rank $\tilde{r} \leq r$.
				\item[] {\bf (Lemma~\ref{lemma: SDP unique solution}-C2)} $\mathcal{N}(\mathcal{A}_V) = \{0_{r \times r}\}$
			\end{itemize}
			\vspace{.05in}
			If conditions (Lemma~\ref{lemma: SDP unique solution}-C1) and (Lemma~\ref{lemma: SDP unique solution}-C2) hold then $X$ is the unique solution of~\eqref{prob: semidefinite relaxation}.	
	\end{lemma} }
	\vspace{.1in}
	
	
	\begin{proof}[Proof of \Cref{thm: necessary and sufficient condition for strong duality}]
		
		We will start by proving sufficiency, and later prove necessity.
		
		\noindent \textbf{(Proof of the sufficiency portion of \Cref{thm: necessary and sufficient condition for strong duality})}: \\
		In proving sufficiency, we start by assuming that $\hat{X}$ and $(\hat{Z},\hat{y}_0,\hat{y}_1,\hat{y}_2)$ is a pair of optimal solutions for the semidefinite relaxation \eqref{prob: semidefinite relaxation} and dual problem \eqref{prob: dual problem}, respectively, and that they satisfy \textit{Property~I$^+$}.
		\vspace{.1in} \\
		\noindent \underline{Main idea:} The main tenet of our proof of sufficiency is to establish the fact that $\hat{X}$ is a unique solution of \eqref{prob: semidefinite relaxation}. Notice that since $\hat{X}$ satisfies \textit{Property~I$^+$}, and hence has rank $2$, the aforementioned fact implies that \eqref{prob: semidefinite relaxation} has no rank-one optimal solution, in which case an optimality gap exists. Indeed, if there was no optimality gap then one could construct a rank-one optimal solution for \eqref{prob: semidefinite relaxation} from that of~\eqref{prob: original  problem}.
		\vspace{.1in} \\ {\firstRevisionColor
			\noindent \underline{Proving uniqueness of $\hat{X}$:} We prove uniqueness by establishing that the conditions of Lemma~\ref{lemma: SDP unique solution} hold for $r=2$.} \\
		
		\noindent \underline{Establishing (Lemma~7.1-C1):} 
		Let $\tilde{X}$ denote any optimal solution of \eqref{prob: semidefinite relaxation}. Using Sylvester's inequality and the complementary condition \eqref{eq: complementary condition-XZ} we conclude that the following holds: 
		\begin{equation}\label{eq: max rank of SP is 2}
		\rank(\tilde{X}) + \rank(\hat{Z}) - (n+1) \leq \rank(\tilde{X}\hat{Z}) = 0 \Rightarrow \rank(\tilde{X}) \leq 2,
		\end{equation}
		that is, the maximum rank of an optimal solution of \eqref{prob: semidefinite relaxation} is 2.
		\vspace{.1in} \\
		\noindent \underline{Establishing (Lemma~7.1-C2):} {\SecondRevisionColor We use the fact that (\textit{I$^+$.4}) is assumed to hold to select a rank-one decomposition based on $(\hat{x}_1,\hat{x}_2)$ satisfying (\ref{eq: equality on the first constraint})--(\ref{eq:I+4Cond}). Using this choice, we form $V$ such that $\hat{X}=VV^T$ and we also construct a matrix $\Delta$ as follows:}
		\begin{equation}\label{eq: a partition of V x1 x2 and Delta}
		V = \begin{bmatrix}
		\hat{x}_1 &  \hat{x}_2
		\end{bmatrix},\quad \hat{x}_1 = \begin{bmatrix}
		\hat{t}_1 \\
		\hat{z}_1
		\end{bmatrix}, \quad
		\hat{x}_2 = \begin{bmatrix}
		\hat{t}_2 \\
		\hat{z}_2
		\end{bmatrix},\quad \Delta = \begin{bmatrix}
		\Delta_1 & \Delta_2 \\
		\Delta_2 & \Delta_3
		\end{bmatrix},
		\end{equation}
		where $\hat{z}_1$ and $\hat{z}_2$ are $n$-dimensional real vectors and $\hat{t}_1,\hat{t}_2,\Delta_1,\Delta_2$, and $\Delta_3$ are real numbers.
		Next, we show that $\mathcal{N}(\mathcal{A}_V) = \{0_{2 \times 2}\}$ holds. Equivalently, we need to show that the only solution $\Delta$ in $\mathcal{S}^2$ for the following system of equations is $\Delta = 0_{2 \times 2}$:
		\begin{equation} \label{eq: linear equation to solve}
		\begin{bmatrix}
		(V^T M(q_1)V) \bigcdot \Delta \\
		(V^T M(q_2)V)\bigcdot \Delta \\
		(V^T I_{00} V) \bigcdot \Delta \\
		\end{bmatrix}=0_{3 \times 1}.
		\end{equation}
		{\SecondRevisionColor From the fact that (\textit{I.1}) is true, we invoke~$\hat{y}_1\hat{y}_2 \neq 0$ in connection with~(\ref{eq: complementary condition-q2}) to conclude that the following holds:}
		\begin{equation}
		M(q_2)\bigcdot \hat{X} = M(q_2) \bigcdot \hat{x}_1\hat{x}_1^T + M(q_2) \bigcdot \hat{x}_2 \hat{x}_2^T = 0. 
		\end{equation}
		We proceed to define $\alpha$ as follows: 
		\begin{equation}
		\alpha := M(q_2) \bigcdot \hat{x}_1 \hat{x}_1^T = -M(q_2) \bigcdot \hat{x}_2\hat{x}_2^T
		\end{equation}
		and
		\begin{equation*}
		\Gamma \defeq \begin{bmatrix}
		M(q_1) \bigcdot \hat{x}_1 \hat{x}_1^T & 2 M(q_1) \bigcdot \hat{x}_1\hat{x}_2^T & M(q_1) \bigcdot \hat{x}_2 \hat{x}_2^T \\
		\alpha & 2 M(q_2) \bigcdot \hat{x}_1\hat{x}_2^T & -\alpha \\
		\hat{t}_1^2 & 2 \hat{t}_1 \hat{t}_2 & \hat{t}_2^2
		\end{bmatrix} \overset{(a)}{=} \begin{bmatrix}
		0 & 2 M(q_1) \bigcdot \hat{x}_1\hat{x}_2^T & 0 \\
		\alpha & 2 M(q_2) \bigcdot \hat{x}_1\hat{x}_2^T & -\alpha \\
		\hat{t}_1^2 & 2 \hat{t}_1 \hat{t}_2 & \hat{t}_2^2
		\end{bmatrix},
		\end{equation*}
		where $(a)$ above follows from~\eqref{eq: equality on the first constraint}.
		We proceed by using $\Gamma$ to express~\eqref{eq: linear equation to solve} as a system of linear equalities for $\Delta_1,\Delta_2$ and $\Delta_3$:
		\begin{equation}\label{eq: linear equation of Delta1 2 3}
		\begin{bmatrix}
		V^T M(q_1)V \bigcdot \Delta \\
		V^T M(q_2)V \bigcdot \Delta \\
		V^T I_{00} V \bigcdot \Delta \\
		\end{bmatrix}=0_{3 \times 1} \Longleftrightarrow
		\Gamma \begin{bmatrix}
		\Delta_1 \\
		\Delta_2 \\
		\Delta_3
		\end{bmatrix} = 0_{3 \times 1}.
		\end{equation}
		Notice that $I_{00}\bigcdot \hat{X} = 1$ implies $\hat{t}_1^2+\hat{t}_2^2 = 1$, which causes the matrix $\Gamma$ to be full rank because
		\begin{equation}
		\det(\Gamma) = -(2M(q_1)\bigcdot \hat{x}_1\hat{x}_2^T)(\alpha (\hat{t}_1^2 +\hat{t}_2^2)) \neq 0.
		\end{equation} {\SecondRevisionColor Here, we also used the fact that $(\hat{x}_1,\hat{x}_2)$ satisfying~(\ref{eq: inequality on the second constraint}) and~(\ref{eq:I+4Cond}) implies that $\alpha \neq 0$ and $M(q_1)\bigcdot \hat{x}_1\hat{x}_2^T \neq 0$, respectively.}
		So $\Delta_1,\Delta_2$, and $\Delta_3$ are all zero and the only solution of \eqref{eq: linear equation to solve} is $\Delta = 0_{2 \times 2}$.
		\vspace{.1in} \\
		\noindent \underline{This concludes our proof of the sufficient condition.} 
		\vspace{.1in} \\
		

		\tikzstyle{startstop} = [rectangle, rounded corners, minimum width=3cm, minimum height=1cm,text centered, draw=black, fill=red!30]
		\tikzstyle{io} = [trapezium, trapezium left angle=70, trapezium right angle=110, minimum width=3cm, minimum height=0.5cm, text centered, draw=black, thick, fill=blue!7]
		\tikzstyle{case} = [rectangle, minimum width=1.5cm, minimum height=0.5cm, text centered, draw=black, thick, fill=yellow!7]
		\tikzstyle{property} = [rectangle, minimum width=1.5cm, minimum height=0.5cm, text centered, draw=black, thick, fill=orange!7]
		\tikzstyle{decision} = [diamond, text centered, draw=black,thick, fill=green!7]
		\tikzstyle{arrow} = [->,>=stealth, thick]
		
		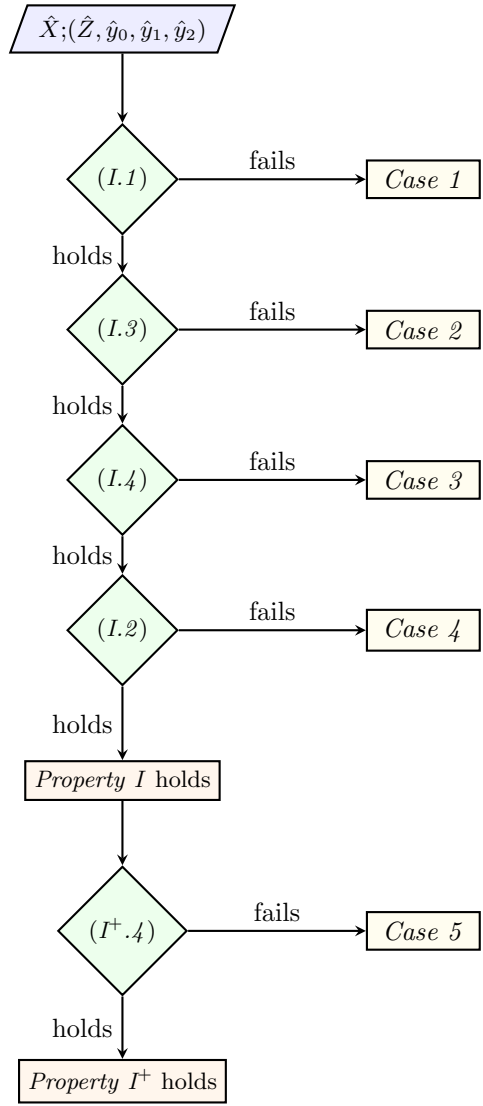
\begin{figure}[h!]
			\begin{center}
				\begin{tikzpicture}[node distance=2cm]
				\node (data) [io] {\small $\hat{X}$;$(\hat{Z},\hat{y}_0,\hat{y}_1,\hat{y}_2)$};
				\node (I1) [decision, below of=data] {\small (\textit{I.1})};
				\node (Case1) [case, right of=I1, xshift=2cm]{\textit{Case 1}};
				
				\node (I3) [decision, below of=I1] {\small (\textit{I.3})};
				\node (Case2) [case, right of=I3, xshift=2cm]{\textit{Case 2}};
				
				\node (I4) [decision, below of=I3] { \small (\textit{I.4})};
				\node (Case3) [case, right of=I4, xshift=2cm]{\textit{Case 3}};
				
				\node (I2) [decision, below of=I4] { \small (\textit{I.2})};
				\node (Case4) [case, right of=I2, xshift=2cm]{\textit{Case 4}};
				
				\node (PropertyI) [property, below of=I2] { \small \textit{Property I} holds};
				
				\node (I4p) [decision, below of=PropertyI] { \small (\textit{I$^+$.4})};
				\node (Case5) [case, right of=I4p, xshift=2cm]{\textit{Case 5}};
				
				\node (PropertyIp) [property, below of=I4p] {\small \textit{Property I$^+$} holds};
				
				\draw [arrow](data) -- (I1);
				\draw [arrow](I1) -- node[anchor=south] {fails} (Case1);
				\draw [arrow](I1) -- node[anchor=east] {holds} (I3);
				
				\draw [arrow](I3) -- node[anchor=south] {fails} (Case2);
				\draw [arrow](I3) -- node[anchor=east] {holds} (I4);
				
				\draw [arrow](I4) -- node[anchor=south] {fails} (Case3);
				\draw [arrow](I4) -- node[anchor=east] {holds} (I2);
				
				\draw [arrow](I2) -- node[anchor=south] {fails} (Case4);
				\draw [arrow](I2) -- node[anchor=east] {holds} (PropertyI);
				
				\draw [arrow](PropertyI) -- (I4p);
				
				\draw [arrow](I4p) -- node[anchor=south] {fails} (Case5);
				\draw [arrow](I4p) -- node[anchor=east] {holds} (PropertyIp);
				
				\end{tikzpicture}
			\end{center}
			\caption{Flowchart describing the cases in the proof of the necessary condition of Theorem~\ref{thm: necessary and sufficient condition for strong duality}.}
			\label{fig:flowchart}
		\end{figure}
		\vspace{.1 in}

		\noindent \textbf{(Proof of the necessity portion of \Cref{thm: necessary and sufficient condition for strong duality})}: \\ 
		In proving necessity, we start by assuming that $\hat{X}$ and $(\hat{Z},\hat{y}_0,\hat{y}_1,\hat{y}_2)$ is a pair of optimal solutions for~\eqref{prob: semidefinite relaxation} and~\eqref{prob: dual problem}, respectively, and that the pair \underline{violates} \textit{Property~I$^+$}.
		
		We proceed to prove in a constructive way that there is no optimality gap. Specifically, we will examine five exhaustive 
		{\firstRevisionColor and mutually exclusive} cases for which \textit{Property~I$^+$} is violated. We will outline in each of these cases, denoted by \textit{Case $1$} through \textit{Case $5$}, how to construct a rank-one solution for~\eqref{prob: semidefinite relaxation} (and hence also~\eqref{prob: original  problem}), thus showing that there is no optimality gap. {\SecondRevisionColor We shall identify each case using the labels ({\it I.1})--({\it I.4}) and ({\it I$^+$.4}) of the conditions specifying \textit{Property~I} and \textit{Property~I$^{+}$} in Definitions~\ref{def: Wenbao's Property I} and~\ref{def:PropertyI+}, respectively. The flowchart in Fig.~\ref{fig:flowchart} summarizes how we identify the cases in terms of the condition labels.}
		
		\vspace{.1 in}
		{\SecondRevisionColor \noindent {\bf(}\underline{\textit{Case 1}}: (\textit{I.1}) fails.{\bf)} }
		
		In this case, $\hat{y}_1\hat{y}_2 = 0$. 
		{\firstRevisionColor We proceed by assuming, without loss of generality, that $\hat{y}_2 = 0$, as the analysis for $\hat{y}_1 = 0$ would be analogous. Since $\hat{y}_2 = 0$, the complementary slackness equality~\eqref{eq: complementary condition-q2} indicates that $\hat{X}$ is only required to satisfy the following feasibility constraint of~\eqref{prob: semidefinite relaxation}:
			\begin{equation} \label{eq:Case-1-Mq2-ineq}
			M(q_2) \bigcdot \hat{X} \leq 0.
			\end{equation}
			If $\hat{y}_1>0$ holds then the complementary slackness condition~\eqref{eq: complementary condition-q1} indicates that $\hat{X}$ satisfies the corresponding feasibility constraint in~\eqref{prob: semidefinite relaxation} with equality, and the case in which~$\hat{y}_1=0$ causes~\eqref{eq: complementary condition-q1} to be satisfied for any feasible solution:
			\begin{equation}
			\left\{
			\begin{aligned}
			M(q_1) \bigcdot \hat{X} \leq 0, \quad \text{if } \hat{y}_1=0 \\
			M(q_1) \bigcdot \hat{X} = 0, \quad \text{if } \hat{y}_1>0
			\end{aligned} \right..
			\end{equation}
			Subsequently, we follow the procedure in the proof of~\cite[Lemma~2.2]{ye2003new} to compute a rank-one decomposition $\hat{X} = \hat{x}_1\hat{x}_1^T + \hat{x}_2\hat{x}_2^T + \dots +\hat{x}_r\hat{x}_r^T$, with $r \defeq \rank(\hat{X})$,
			such that the following holds:
			\begin{equation}\label{eq: two situations of y1 in case 1}
			\left\{
			\begin{aligned}
			M(q_1) \bigcdot \hat{x}_i \hat{x}_i^T \leq 0, \quad \text{if } \hat{y}_1=0 \\
			M(q_1) \bigcdot \hat{x}_i \hat{x}_i^T = 0, \quad \text{if } \hat{y}_1>0
			\end{aligned} \right.
			,\quad  i \in \{1,2, \dots, r \}.
			\end{equation} 
			Now, notice that~\eqref{eq:Case-1-Mq2-ineq} implies that there exists $i$ in $\{1,2, \dots, r \}$ for which $M(q_2) \bigcdot \hat{x}_i \hat{x}_i^T \leq 0$ holds. Hence, we can order the indices of $\{\hat{x}_1, \ldots, \hat{x}_r \}$ so that $\hat{x}_1$ satisfies
			\begin{equation}\label{eq: one vector feasible for second constraint}
			M(q_2) \bigcdot \hat{x}_1 \hat{x}_1^T \leq 0.
			\end{equation}
			We proceed to argue that $\hat{x}_1 \hat{x}_1^T/\hat{t}_1^2$ is an optimal solution of \eqref{prob: semidefinite relaxation}, or, equivalently, that $\hat{x}_1 / \hat{t}_1$ is an optimal solution to \eqref{prob: original  problem}, where $\hat{t}_1$ comes from the partition of $\hat{x}_1$ in \eqref{eq: partition of x}. In order to do so, we notice that $\hat{x}_1$ is in the range space of $\hat{X}$, and, hence, must be in the null space of $\hat{Z}$, which implies $\hat{x}_1 \hat{x}_1^T \hat{Z} = 0_{(n+1) \times (n+1)}$. This fact, together with \eqref{eq: two situations of y1 in case 1} and \eqref{eq: one vector feasible for second constraint}, establishes that $\hat{x}_1 \hat{x}_1^T/\hat{t}_1^2$ satisfies~\eqref{eq: complementary condition}. Hence, we can invoke the complementary slackness principle to argue that $\hat{x}_1 \hat{x}_1^T/\hat{t}_1^2$ is an optimal solution of~\eqref{prob: semidefinite relaxation}. 
			Finally, it remains to show that $\hat{t}_1 \neq 0$. We do so by showing that ${\hat{t}_1 = 0}$ would contradict \Cref{asmp: Slater for SD}, which is Slater's condition for~\eqref{prob: dual problem}. A key observation is that if $\hat{t}_1 = 0$ then $\hat{X} + \alpha \hat{x}_1 \hat{x}_1^T$ is a solution to~\eqref{prob: semidefinite relaxation} satisfying~\eqref{eq: complementary condition} for every $\alpha \geq 0$. Hence, we can use $\hat{Z} \bigcdot \hat{x}_1 \hat{x}_1^T = 0$ to invoke the complementary slackness principle to argue that $\mathbb{X}:=\{\hat{X} + \alpha \hat{x}_1 \hat{x}_1^T \ | \ \alpha \geq 0 \}$ is a set of optimal solutions for~\eqref{prob: semidefinite relaxation} when $\hat{t}_1=0$. Given that $\hat{x}_1 \neq 0$, we conclude that $\mathbb{X}$ is unbounded.
			
			However, according to \cite[Theorem 2.5]{sturm1997primal}, the unboundedness of $\mathbb{X}$ would contradict \Cref{asmp: Slater for SD}. 
		}
		
		\vspace{.1in}
		
		
		\noindent {\SecondRevisionColor {\bf(}\underline{\textit{Case 2}}: (\textit{I.1}) holds and (\textit{I.3}) fails.{\bf)} }
		
		In this case, $\hat{y}_1\hat{y}_2 \neq 0$ and $\rank(\hat{X}) \neq 2$. Notice that $r := \rank(\hat{X})$ is positive because, from~\eqref{prob: semidefinite relaxation}, we know that $I_{00} \bigcdot \hat{X} = 1$ holds. If $r=1$ then $\hat{X} = \hat{x}\hat{x}^T$ is already a rank-one optimal solution to~\eqref{prob: semidefinite relaxation}, or equivalently, $\hat{x}$ is an optimal solution to~\eqref{prob: original  problem}. 
		{\firstRevisionColor Hence, subsequently, we proceed to analyze the case in which~$r \geq 3$.  
			
			We start by invoking the complementary condition~\eqref{eq: complementary condition} to conclude that $\hat{y}_1\hat{y}_2 \neq 0$ implies:
			\begin{equation}
			M(q_1) \bigcdot \hat{X}= M(q_2) \bigcdot \hat{X}=0.
			\end{equation}
			Hence, given that $r \geq 3$, by \cite[Theorem~3.4]{ai2009strong}, one can find a rank-one decomposition ${\hat{X} = \hat{x}_1\hat{x}_1^T+\hat{x}_2\hat{x}_2^T+\dots+\hat{x}_r\hat{x}_r^T}$ such that the following holds:
			\begin{subequations}\label{eq: case 2 feasible rank-one decomposition}
				\begin{align}
				M(q_1) \bigcdot \hat{x}_i \hat{x}_i^T & = 0, \qquad i \in \{1,2,\dots,r \}, \\
				M(q_2) \bigcdot \hat{x}_j \hat{x}_j^T & = 0, \qquad j \in \{1,2,\dots,r-2 \}.
				\end{align}
			\end{subequations}
			
			We proceed to argue that $\hat{x}_1 \hat{x}_1^T/\hat{t}_1^2$ is an optimal solution of \eqref{prob: semidefinite relaxation}, or, equivalently, that $\hat{x}_1 / \hat{t}_1$ is an optimal solution to \eqref{prob: original  problem}, where $\hat{t}_1$ comes from the partition of $\hat{x}_1$ in \eqref{eq: partition of x}. In order to do so, we notice that $\hat{x}_1$ is in the range space of $\hat{X}$, and, hence, must be in the null space of $\hat{Z}$, which implies $\hat{x}_1 \hat{x}_1^T \hat{Z} = 0_{(n+1) \times (n+1)}$. This fact, together with~\eqref{eq: case 2 feasible rank-one decomposition}, establishes that $\hat{x}_1 \hat{x}_1^T/\hat{t}_1^2$ satisfies~\eqref{eq: complementary condition}. Hence, we can invoke the complementary slackness principle to argue that $\hat{x}_1 \hat{x}_1^T/\hat{t}_1^2$ is an optimal solution of~\eqref{prob: semidefinite relaxation}.

			
			Finally, we proceed to prove by contradiction that $\hat{t}_1 \neq 0$. Our strategy is to show that assuming $\hat{t}_1 = 0$ would contradict Slater's condition~\eqref{eq: Slater's condition for SD 1}. 
			Hence, assuming that $\hat{t}_1 = 0$ and recalling that $\hat{x}_1^T \hat{Z} \hat{x}_1 = \Tr(\hat{x}_1 \hat{x}_1^T \hat{Z})  =0$, we can use the first constraint in~\eqref{prob: dual problem} to write:
			\begin{multline}\label{eq: contradiction 1}
			\hat{x}_1^T \Big( -\hat{y}_0 I_{00} + M(q_0) + \hat{y}_1 M(q_1) + \hat{y}_2 M(q_2)   \Big) \hat{x}_1 = \\ 
			\hat{z}_1^T(Q_0+\hat{y}_1Q_1+\hat{y}_2Q_2)\hat{z}_1 =  0,
			\end{multline}
			where $\hat{z}_1$ comes from the partition of $\hat{x}_1$ in \eqref{eq: partition of x}. On the other hand, when we substitute $\hat{t}_1=0$ in~\eqref{eq: case 2 feasible rank-one decomposition}, we get:
			\begin{equation}\label{eq: contradiction 2}
			\hat{z}_1^T Q_1 \hat{z}_1 = \hat{z}_1 ^T Q_2 \hat{z}_1 = 0.
			\end{equation}
			Combining \eqref{eq: contradiction 1} and \eqref{eq: contradiction 2}, we conclude that:
			\begin{equation}\label{eq:cond2-contradiction-3}
			\hat{z}_1^T Q_0 \hat{z}_1 = 0.
			\end{equation}
			However, \eqref{eq: contradiction 1}--\eqref{eq:cond2-contradiction-3} would imply that the following would hold for all $\bar{y}_1,\bar{y}_2 > 0$:
			\begin{equation}
			z_1^T (Q_0 + \bar{y}_1 Q_1 + \bar{y}_2 Q_2) z_1 = 0,
			\end{equation}
			which would contradict \Cref{asmp: Slater for SD}, unless $\hat{z}_1 = 0_{n \times 1}$ or, equivalently, $\hat{x}_1 = 0_{(n+1) \times 1}$ because we are assuming $\hat{t}_1=0$. However, the definition of rank-one decomposability \cite{ai2009strong} guarantees that $\hat{x}_1 \neq 0_{(n+1) \times 1}$. Hence, we reached a contradiction, from which we conclude that $\hat{t}_1 \neq 0$.
		}     
		\vspace{.1in}
		
		{\SecondRevisionColor
			\begin{remark}\label{rem:setO}  Before we proceed our analysis of the subsequent cases, we introduce the following sets:
				\begin{align} 
				\mathbb{O}(\hat{X}) &:= \{(\hat{x}_1,\hat{x}_2) \in \mathbb{R}^{n+1}\times \mathbb{R}^{n+1} \ | \text{ $\hat{X} = \hat{x}_1\hat{x}_1^T+ \hat{x}_2\hat{x}_2^T$ and (\ref{eq: equality on the first constraint}) holds}\}, \\
				\mathbb{O}^+(\hat{X}) &:= \{(\hat{x}_1,\hat{x}_2) \in \mathbb{O}(\hat{X}) \ | \text{  (\ref{eq: inequality on the second constraint}) holds}\}.
				\end{align}
				From \cite[Lemma~2.2]{ye2003new}, we conclude that when $\hat{X}$ is a solution of~\eqref{prob: semidefinite relaxation} satisfying $\rank(\hat{X}) = 2$ and $M(q_1) \bigcdot \hat{X} = 0$, it is always possible to obtain a rank-one decomposition $\hat{X} = \hat{x}_1 \hat{x}_1^T + \hat{x}_2 \hat{x}_2^T$ that complies with~\eqref{eq: equality on the first constraint}. This fact and~(\ref{eq: complementary condition-q1}) imply that $\mathbb{O}(\hat{X})$ is nonempty when $\hat{y}_1 \neq 0$ and $\rank(\hat{X})$ is $2$. Hence, we can infer that if (\textit{I.1}) and (\textit{I.3}) hold then $\mathbb{O}(\hat{X})$ is nonempty and the cases in which ({\it I.4}) fails or holds have the following further implications in the terms of $\mathbb{O}^+(\hat{X})$:
				\begin{itemize}
					\item If  (\textit{I.1}) and (\textit{I.3}) hold, and ({\it I.4}) fails, then $\mathbb{O}(\hat{X})$ is nonempty, but $\mathbb{O}^+(\hat{X})$ is empty, or, equivalently, (\ref{eq: inequality on the second constraint}) is violated for every pair $(\hat{x}_1,\hat{x}_2)$ in $\mathbb{O}(\hat{X})$. 
					\item If (\textit{I.1}), (\textit{I.3}), and ({\it I.4}) hold then $\mathbb{O}^+(\hat{X})$ is nonempty. 
				\end{itemize}
			\end{remark}
		}
		
		\vspace{.1 in} 
		\noindent 
		{\SecondRevisionColor {\bf(}\underline{\textit{Case 3}}: (\textit{I.1}) and (\textit{I.3}) hold, and ({\it I.4}) fails.{\bf)} 
			
			In this case, Remark~\ref{rem:setO} guarantees that $\mathbb{O}(\hat{X})$ is nonempty but~(\ref{eq: inequality on the second constraint}) is violated for every pair in it, which, together with~(\ref{eq: complementary condition-q2}) and the fact that $\hat{y}_2 \neq 0$, imply that: 
			\begin{equation}
			\label{eq:case3-1} M(q_2) \bigcdot \hat{x}_1\hat{x}_1^T = M(q_2) \bigcdot \hat{x}_2\hat{x}_2^T = 0, \quad (\hat{x}_1,\hat{x}_2) \in \mathbb{O}(\hat{X}).
			\end{equation}
			
			We start by selecting $(\hat{x}_1,\hat{x}_2)$ in $\mathbb{O}(\hat{X})$ and adopting the partition of $\hat{x}_1$ and $\hat{x}_2$ in \eqref{eq: a partition of V x1 x2 and Delta}.
			Consequently, $I_{00} \bigcdot \hat{X} = 1$ implies that $\hat{t}_1^2 + \hat{t}_2^2 = 1$, that is, at least one of $\hat{t}_1$ and $\hat{t}_2$ is nonzero. Without loss of generality, assume $\hat{t}_1 \neq 0$. 
			Since $\hat{x}_1$ is in the range space of $\hat{X}$, (\ref{eq: complementary condition-XZ}) implies that it must be in the null space of~$\hat{Z}$, leading to $\hat{x}_1 \hat{x}_1^T \hat{Z} = 0_{(n+1) \times (n+1)}$. Now, recall that because it is an element of $\mathbb{O}(\hat{X})$, $(\hat{x}_1,\hat{x}_2)$ satisfies $M(q_1) \bigcdot \hat{x}_1 \hat{x}_1^T=0$. Finally, using the facts that $\hat{x}_1 \hat{x}_1^T \hat{Z} = 0_{(n+1) \times (n+1)}$,  $M(q_1) \bigcdot \hat{x}_1 \hat{x}_1^T = 0$ and~(\ref{eq:case3-1}) we can invoke~(\ref{eq: complementary condition}) to conclude that $\hat{x}_1\hat{x}_1^T / \hat{t}_1^2$ is a rank-one optimal solution of \eqref{prob: semidefinite relaxation}.
			Therefore, $\hat{x}_1 / \hat{t}_1$ is an optimal solution of \eqref{prob: original  problem}. }
		\vspace{.1in}
		
		\noindent {\SecondRevisionColor {\bf(}\underline{\textit{Case 4}}: (\textit{I.1}), (\textit{I.3}), and (\textit{I.4}) hold, and (\textit{I.2}) fails{\bf)}
			
			In this case, $\hat{y}_1\hat{y}_2 \neq 0$, $\rank(\hat{Z}) \neq n-1$, $\rank(\hat{X}) = 2$, and, by Remark~\ref{rem:setO}, $\mathbb{O}^+(\hat{X})$ is nonempty. }
		
		We start by noticing that the following holds:
		\begin{align}
		\rank(\hat{Z})+\rank(\hat{X}) & \leq n+1,\\
		\rank(\hat{X}) & = 2,\\
		\rank(\hat{Z}) & \neq n-1,
		\end{align}
		which implies that the rank condition below is statisfied
		\begin{equation}
		\rank(\hat{Z})<n-1,
		\end{equation}
		indicating that the following inequality holds:
		\begin{equation}
		\rank(\hat{Z}) + \rank(\hat{X}) < n+1.
		\end{equation}
		Now $\hat{X}+ \hat{Z}$ is singular because $\rank(\hat{X}+\hat{Z}) \leq \rank(\hat{X}) + \rank(\hat{Z})$. Also, both $\hat{X}$ and $\hat{Z}$ are positive semidefinite. So there must be a nontrivial $(n+1)$-dimensional real vector $y$ in the intersection of the null space of $\hat{X}$ and the null space of $\hat{Z}$. {\SecondRevisionColor Select a pair $(\hat{x}_1,\hat{x}_2)$ in $\mathbb{O}^+(\hat{X})$ and let }
		\begin{equation}
		X \defeq \hat{X} + yy^T = \hat{x}_1\hat{x}_1^T + \hat{x}_2\hat{x}_2^T + yy^T.
		\end{equation}
		Consequently, $\rank(X) = 3$ and $X \hat{Z} = 0_{(n+1)\times (n+1)}$ because $\hat{X} \hat{Z} = 0_{(n+1)\times (n+1)}$ and $y$ is in the null space of $\hat{Z}$. By \cite[Lemma~3.3]{ai2009strong}, we know there exists an $(n+1)$-dimensional real vector $x$ such that $X$ is rank-one decomposable at $x$ and that 
		\begin{equation}\label{eq: existence of rank-one decomposable x}
		M(q_1) \bigcdot xx^T = M(q_2)\bigcdot xx^T = 0.
		\end{equation}
		
		Since $x$ is in the range space of $X$, it must be in the null space of $\hat{Z}$ which implies $xx^T \hat{Z} = 0_{(n+1) \times (n+1)}$. Therefore, by the complementary condition \eqref{eq: complementary condition-XZ}, $xx^T \hat{Z} = 0_{(n+1) \times (n+1)}$ and \eqref{eq: existence of rank-one decomposable x} imply that $xx^T/t^2$ is an optimal solution of \eqref{prob: semidefinite relaxation}, where $t$ comes from the partition of $x$ in \eqref{eq: partition of x}. Hence, $x/t$ is an optimal solution of \eqref{prob: original  problem}. Note that $t \neq 0$ follows the same argument of $\hat{t}_1 \neq 0$ in \textit{Case 2}.

		\vspace{.1in}
		
		\noindent {\SecondRevisionColor {\bf(}\underline{\textit{Case 5}}: (\textit{I.1})--(\textit{I.4}) hold and ({\it I$^+$.4}) fails.{\bf)}
			
			This is the case in which $\hat{X}$ and $(\hat{Z},\hat{y}_0,\hat{y}_1,\hat{y}_2)$ violate \textit{Property~I$^+$}, while satisfying \textit{Property~I}. Notice that there is no case comparable to this in~\cite{ai2009strong}.
			
			We start by observing that since (\textit{I.1})--(\textit{I.4}) hold by assumption, it follows that $\hat{y}_1\hat{y}_2 \neq 0$, $\rank(\hat{Z}) = n-1$, $\rank(\hat{X}) = 2$, and Remark~\ref{rem:setO} guarantees that $\mathbb{O}^+(\hat{X})$ is nonempty. Furthermore, under the stated assumption that ({\it I$^+$.4}) fails, (\ref{eq:I+4Cond}) is violated for all pairs in $\mathbb{O}^+(\hat{X})$. Hence, the following holds:
			\begin{equation}
			\label{eq:case5-1}
			M(q_1) \bigcdot \hat{x}_1\hat{x}_2^T  = 0, \quad (\hat{x}_1,\hat{x}_2) \in \mathbb{O}^+(\hat{X}).
			\end{equation}
			
			We proceed by selecting $(\hat{x}_1,\hat{x}_2)$ in $\mathbb{O}^+(\hat{X})$.  Our strategy is to use such a choice to construct another rank-one decomposition $\hat{X} = \check{x}_1\check{x}_1^T + \check{x}_2\check{x}_2^T$ satisfying the following equalities:
			\begin{equation}
			\label{eq:Case5Objective}
			M(q_1) \bigcdot \check{x}_i\check{x}_i^T = M(q_2) \bigcdot \check{x}_i \check{x}_i^T = 0,\qquad i \in \{1,2\}.
			\end{equation} 
			Because (\ref{eq:Case5Objective}) will hold by construction, $(\check{x}_1,\check{x}_2)$ will also satisfy~(\ref{eq:case3-1}). Consequently, once the construction of $(\check{x}_1,\check{x}_2)$ is complete, we can employ the method used in {\it Case~$3$} to compute an optimal solution of~\eqref{prob: original  problem}.
			
			Hence, it remains to construct a pair $(\check{x}_1,\check{x}_2)$ that satisfies~(\ref{eq:Case5Objective}). We start by noticing that~(\ref{eq:case5-1}) together with~(\ref{eq: equality on the first constraint}), which is satisfied for our choice $(\hat{x}_1,\hat{x}_2)$ in $\mathbb{O}^+(\hat{X})$, imply that the following holds:
			\begin{equation}
			M(q_1) \bigcdot (\alpha_1 \hat{x}_1 + \alpha_2 \hat{x}_2)(\alpha_1 \hat{x}_1 + \alpha_2 \hat{x}_2)^T = 0, \quad \alpha_1,\alpha_2 \in \mathbb{R}.
			\end{equation}
			
		} 
		
		The final step is to construct $\check{x}_1$ and $\check{x}_2$ as linear combinations of $\hat{x}_1$ and $\hat{x}_2$ so that the following equality holds:
		\begin{equation}
		{M(q_2) \bigcdot \check{x}_i \check{x}_i^T = 0}, \qquad i\in \{1,2\}.
		\end{equation}
		
		Since from~\eqref{eq: complementary condition-q2} and $\hat{y}_2\neq 0$ we know that $M(q_2) \bigcdot \hat{X} = 0$, we proceed to determine such $\check{x}_1$ and $\check{x}_2$ by following the procedure in the proof of~\cite[Lemma~2.2]{ye2003new}. To do so, we substitute $G$, $u_i$ and $\bar{u}_i$ present in~\cite[p.~249]{ye2003new} with $M(q_2)$, $\hat{x}_i$ and $\Check{x}_i$, respectively, for $i$ in $\{1,2\}$.
		The rank-one decomposition $\hat{X} =  \check{x}_1\check{x}_1^T + \check{x}_2\check{x}_2^T$ obtained in this way will satisfy~\eqref{eq:Case5Objective}.
		\vspace{.1in} \\
		\noindent \underline{This concludes our proof of the necessary condition.}
	\end{proof} 
	
	\vspace{.1in} 
	
	\begin{remark}
		The proof of \cite[Theorem~2.6]{yuan2017new} also shows the uniqueness of the solution of \eqref{prob: semidefinite relaxation} in the CDT subproblem when \textit{Property~I} holds. The proof uses a property of the boundary points of an SOC whereas we use a result on the uniqueness of the solution of a semidefinite program.
	\end{remark}
	
	\section*{Acknowledgments} The authors would like to thank Dr. MirSaleh Bahavarnia for reading this paper carefully and providing helpful suggestions.
	
	\bibliographystyle{siamplain}
	\bibliography{reference}
	
\end{document}